%% file: article.tex
\documentclass[leqno]{scrartcl}

\usepackage{amsmath, amssymb, color, graphicx} % Math packages
\usepackage{caption}
\usepackage{amsthm}
\usepackage{subcaption}
\usepackage{enumitem}

\numberwithin{equation}{section}

\newcommand{\ds}{\ensuremath{\displaystyle}}

\newtheorem{proposition}{Proposition}
\newtheorem{lemma}{Lemma}
\newtheorem{claim}{Claim}

\theoremstyle{remark}
\newtheorem{remark}{Remark}

\begin{document}

%+Title
\title{Roy's largest root under rank-one alternatives: The complex valued case and applications }
\author{ Prathapasinghe Dharmawansa\and
Boaz Nadler\and
  Ofer Shwartz
}
\date{\today}
\maketitle
%-Title

%+Abstract
\input{abstract}

%--------------------
% INTRODUCTION
% -------------------
\input{intro}

%----------------------------------------------------------------------------------------
\section{On the Distribution of Roy's Largest Root}
\label{sec:propositions}
%----------------------------------------------------------------------------------------

Propositions 1-5 below provide simple approximations to the distribution of Roy's largest root under a rank-one alternative and several common single matrix and double matrix settings. The following 5 propositions correspond to the five cases in Table \ref{table:4_cases}, and extend to the complex-valued setting  propositions 1-5 of \cite{Johnstone_Nadler}.

We start with the simplest setting of a single central Wishart matrix.
Since in cases 1 and 2 the matrix $\Sigma$ is assumed to be known, we assume w.l.g. that $\Sigma=\sigma^2 I$, where $\sigma^2$ is a small parameter, typically representing the noise variance in the absence of a signal. 
In contrast to previous asymptotic approaches whereby $n_H\to\infty$, $m\to\infty$ or both, in the following, we keep the number of samples \(n_{H}\) and the dimension \(m\) fixed, and study the distribution of the largest eigenvalue as \(\sigma\to 0\). We start with the central setting, case 1 in Table \ref{table:4_cases}. 

\begin{proposition}\label{prop:1}
Let $H \sim \mathcal{CW}_m(n, \lambda {\bf v}{\bf v}^\dagger+\sigma^2I)$, with $||{\bf v}||=1$, $\lambda>0$ and let $\ell_1$ be its largest eigenvalue. Then, with $(m, n, \lambda)$ fixed, as $\sigma \rightarrow 0$
\begin{equation}
\label{eq:prop_1}
\ell_1(\sigma) = \frac{\lambda + \sigma^2 }{2}A + \frac{\sigma^2}{2}B + \frac{\sigma^4 }{2(\lambda + \sigma^2)}\frac{B\cdot C}{A} + o(\sigma^4)
\end{equation}
where \(A,B,C\) are independent random variables, distributed as\  $A \sim \chi_{2n}^2, B \sim\chi_{2m-2}^2,$ and $C \sim\chi_{2n-2}^2$.
\end{proposition}

\begin{remark}
Approximate expressions for the mean and variance of \(\ell_1\) follow
directly from Eq. (\ref{eq:prop_1}). Since $\mathbb{E}[\chi^2_{k}]=k$ and for $k>2$
$\mathbb{E}[1/\chi^2_k] = 1/(k-2)$, then for $n>1$
as $\sigma\rightarrow0$, % assuming $n>2 $ NEED TO RECHECK CALCULATIONS FOR REMARKS
\begin{equation}
\label{eq:prop_1_mean}
\mathbb{E}[ \ell_1(\sigma) ]= n \lambda + (n+m-1)\sigma^2  + \frac{\sigma^4 }{\lambda + \sigma^2}(m-1)+o(\sigma^4)
\end{equation}
Similarly, since $\mathrm{Var}[\chi^2_{k}]=2k$ and for $k>4$ 
$\mathrm{Var}[1/\chi^2_k] = \frac{2}{(k-2)^2(k-4)}$, then for $n>3$  
\begin{equation}
\label{eq:prop_1_variance}
\mathrm{Var}[ \ell_1(\sigma)] = 2\left( \lambda n + \sigma^2(n+m-1)+\frac{\sigma^4}{\lambda+\sigma^2}
\frac{(m-1)}{(n-1)(n-2)}\right)+o(\sigma^2)
\end{equation}
\end{remark}

The next proposition considers the non-central single Wishart matrix.
\begin{proposition}\label{prop:2}
Let $H \sim \mathcal{CW}_m(n, \sigma^2I, (\omega / \sigma^2) {\bf v}{\bf v}^\dagger)$, with $||{\bf v}||=1$, and let $\ell_1$ be its largest eigenvalue. Then, with $(m, n, \omega)$ fixed, as $\sigma \rightarrow 0$
\begin{equation}
\label{eq:prop_2}
\ell_1(\sigma) = \frac{\sigma^2}{2}\left[A +B+ \frac{B \cdot C }{A}  \right]+ o(\sigma^4)
\end{equation}
where \(A,B,C\) are all independent and distributed as $A \sim \chi_{2n}^2(2\omega / \sigma^2), B \sim\chi_{2m-2}^2$ and $C \sim\chi_{2n-2}^2$.
\end{proposition}
\begin{remark}
By representing $\chi_k^2(\delta)$ as $\chi_{k+2K}^2$ for $K \sim Poisson(\delta/2)$, we obtain $\mathbb{E[}\chi_k^2(\delta)] = k + \delta$, $\mathbb{E[}(\chi_k^2(\delta))^{-1}]\approx(n-2+\delta)^{-1}$ and $\mathrm{Var}(\chi_k^2(\delta))^{-1}\approx\frac{2}{(k+\delta-2)^2(k+\delta-4)}$. Therefore, assuming $n>2$, we can approximate the expectation and the variance of (\ref{eq:prop_2}) by
\begin{equation}
\label{eq:prop_2_mean}
\mathbb{E}[ \ell_1(\sigma)] \approx (n+m-1)\sigma^2+\omega + \frac{(n-1)(m-1)}{\sigma^2(n-1)+\omega}
\end{equation}
and
\begin{equation}
\label{eq:prop_2_variance}
\mathrm{Var}[ \ell_1(\sigma)] \approx 8\omega + 4\sigma^2\left(  n + m-1+\frac{(n-1)(m-1)}{2(n+\frac{\sigma^2}{\omega}-1)^{2}(n+\frac{\sigma^2}{\omega}-2)}\right)
\end{equation}
\end{remark}

\begin{remark} In both Eqs. (\ref{eq:prop_1}) and (\ref{eq:prop_2}), note that as $n \rightarrow \infty$, the \(\chi^2 \) random variables converge to Gaussian ones, and thus asymptotically the distribution of  $\ell_1$ converges to a Gaussian
as well. This is in accordance with classical asymptotic results, see \cite{Anderson}.
\end{remark}

The next two propositions provide approximations to the distribution of Roy's
largest root in the central and non-central double matrix settings, which correspond to cases 3 and 4 in Table \ref{table:4_cases}. Since $\ell^{-1}(E^{-1}H) = \ell^{-1}((\Sigma^{-1}E)^{-1}(\Sigma^{-1}H))$, we can assume w.l.g. that $\Sigma=I$.
\begin{proposition}\label{prop:3}
Suppose that $H \sim \mathcal{CW}_m(n_H, I+ \lambda {\bf v} {\bf v}^\dagger)$ and $E \sim CW_m(n_E, I)$ are independent complex Wishart matrices, with $n_{E} > m+1$ and $||{\bf v}|| = 1$. Let $\ell_1$ be the largest eigenvalue of $E^{-1}H$. Then, with $(m, n_H, n_E)$ fixed, as $\lambda \rightarrow \infty$
\begin{equation}
\label{eq:prop_3}
\ell_1(\lambda) \approx
(1+\lambda)a_{1}
 F_{b_1,c_1}+
 a_2 F_{b_2,c_2}+
a_3
\end{equation}
where the two $F$ distributed random variates are independent, and
\begin{equation}\label{eq:prop_3_params}
a_1 = \frac{n_H}{n_E-m+1} \qquad
a_2 = \frac{m-1}{n_E-m+2} \qquad
a_3 =\frac{m-1}{(n_E -m)(n_E -m -1)}
\end{equation}
$$
b_1 = 2n_H \qquad
b_2 = 2m-2 \qquad
$$
$$
c_1 = 2n_E-2m+2 \qquad
c_2 = 2n_E-2m+4 \qquad
$$
\end{proposition}

%Similarly, for the non-central case we have the following result.

\begin{proposition}\label{prop:4}
Suppose that $H \sim \mathcal{CW}_m(n_H, I, \omega {\bf v}{\bf v}^\dagger)$ and $E \sim \mathcal{CW}_m(n_E, I)$ are independent complex Wishart matrices, with $n_{E} > m+1$, $\omega>0$ and $||{\bf v}|| = 1$. Let $\ell_1$ be the largest eigenvalue of $E^{-1}H$. Then, with $(m, n_H, n_E)$ fixed, as $\omega \rightarrow \infty$
\begin{equation}
\label{eq:prop_4}
 \ell_1(\omega) \approx
a_1 F_{b_1, c_1}(2\omega)+
a_2 F_{b_2, c_2}+
a_3
\end{equation}
where the two $F$ distributed random variates are independent and the parameters $a_i, b_i, c_i$ are as defined in Eq. (\ref{eq:prop_3_params}).
\end{proposition}

\begin{remark}
In the limit as $n_E \rightarrow \infty$, the two F-distributed random variables in (\ref{eq:prop_3}) and (\ref{eq:prop_4}) converge to $\chi^2$ 
distributed random variables, thus recovering the leading order terms in (\ref{eq:prop_1}) 
and (\ref{eq:prop_2}), respectively.
\end{remark}

%----------------------------------------------------------------------------------------------------
\subsection{On the leading canonical correlation coefficient}
Let $\mathbf{x}_i\sim\mathcal{CN}\left(0,\Sigma\right),\; i=1,2,\cdots,n+1$ denote complex valued multivariate Gaussian observations on $m=p+q$ variables, where without loss of generality we assume that \(p\leq q\). Let us denote the corresponding sample covariance matrix by $S$. In this subsection we consider the fifth setting of Table \ref{table:4_cases} and study the largest sample canonical correlation coefficient between the first group of \(p\) variables and the second group of \(q\) variables, in the presence of a single large canonical correlation coefficient in the population. 

Now, in the presence of a single large population canonical correlation coefficient, which we denote by $\rho$, one can decompose $\Sigma$ as\footnote{Since the canonical correlation is invariant under nonsingular linear transformations, without loss of generality, we can choose this canonical form for the matrix $\Sigma$ .}
\begin{align}
\Sigma=\left(\begin{array}{cc}
I_p & \tilde P\\
\tilde P^\dagger & I_q
\end{array}\right)
\end{align}
where $\tilde P=\left(P \;\;\;\mathbf{0}_{p\times (q-p)}\right)$ with $P=\text{diag}\left(\rho,0,\cdots,0\right)\in\mathbb{R}^{p\times p}$. Similarly, one can decompose the sample covariance matrix as
\begin{align}
nS=\left(\begin{array}{cc}
Y^\dagger Y & Y^\dagger X\\
X^\dagger Y & X^\dagger X
\end{array}
\right)
\end{align}
where $Y\in\mathbb{C}^{n\times p}$ and $X\in \mathbb{C}^{n\times q}$ represent the first $p$ variables and the remaining $q$ variables, respectively. Clearly, the sample canonical correlation coefficients, $r_1,r_2,\cdots,r_p$ are the positive square roots of the eigenvalues of $(Y^\dagger Y)^{-1}Y^\dagger X(X^\dagger X)^{-1}X^\dagger Y$. Our objective is to analyze the largest sample canonical correlation coefficient under the above setting (i.e., in the presence of a single dominant population canonical correlation coefficient). As shown below, a single dominant population canonical correlation coefficient amounts to having a rank-one non-centrality matrix.

Now we can write the squared sample canonical correlation coefficients (i.e., $r_i^2$) as the roots of the following characteristic equation
\begin{align}
\det\left(r^2 Y^\dagger Y-Y^\dagger Q Y\right)=0
\end{align}
where $Q=X^\dagger\left(X^\dagger X\right)^{-1} X$. For convenience, let us introduce $H=Y^\dagger Q Y$ and $ E=Y^\dagger(I_p-Q) Y$.  With this notation, we obtain the modified characteristic equation as
\begin{align}
\det\left(H-r^2(H+E)\right)=0.
\end{align}
This in turn reveals that the study of the largest root of the above equation is equivalent to study of the largest root of $E^{-1}H$. This fact can be further delineated using the relation $\ell_1=r_1^2/(1-r_1^2)$, where $\ell_1$ is the largest root of $E^{-1}H$.

The following complex analog of a result given in \cite{Johnstone_Nadler} is also important in the sequel
\begin{align}
Y|X\sim \mathcal{CN}\left(X \tilde P^\dagger, \Phi\right)
\end{align}
where $\Phi=I_p-P^2$. Therefore, we conclude that
\begin{align}
\label{HEcondx}
H|X&\sim \mathcal{CW}_p\left(q,\Phi,\Omega\right)\nonumber\\
E|X& \sim \mathcal{CW}_p\left(n-q, \Phi\right)
\end{align}
are independent with the non-centrality matrix given by
\begin{align}
\Omega=\Phi^{-1} \tilde P  X^\dagger X \tilde P ^\dagger=\frac{\rho^2}{1-\rho^2} \left( X^\dagger X\right)_{11}\mathbf{e}_1\mathbf{e}_1^{'}=\omega\; \mathbf{e}_1\mathbf{e}_1^{'}
\end{align}
where $\omega=\frac{\rho^2}{1-\rho^2} \left( X^\dagger X\right)_{11}$. Since $X^\dagger X\sim \mathcal{CW}_q\left(n,I_q\right)$, we have $\left(X^\dagger X\right)_{11}\sim \frac{1}{2}\chi^2_{2n}$.
The following proposition gives the distribution of the largest sample canonical correlation in the presence of a single population canonical correlation.

\begin{proposition}\label{canonical}
Let $\ell_1=\frac{r_1^2}{1-r_1^2}$, where $r_1$ is the largest sample canonical correlation between two groups of size $p\leq q$ computed from $n+1$ i.i.d. observations with $\nu=n-p-q>1$. Then in the presence of a single large population correlation coefficient $\rho$ between the two groups, asymptotically as $\rho\to1$,
\begin{align}
\ell_1\left(E^{-1}H\right)\approx a_1 F_{b_1,c_1}^\chi \left(\frac{\rho^2}{1-\rho^2},2n\right)+a_2F_{b_2,c_2}+a_3
\end{align}
where
\begin{align*}
a_1&=q/(\nu+1),\;\;\; a_2=(p-1)/(\nu+2),\;\;\; a_3=(p-1)/\nu(\nu-1),\\
b_1&=2q,\;\;\; b_2=2p-2,\;\; c_1=2(\nu+1),\;\;\; c_2=2(\nu+2).
\end{align*}
Here we have used the notation $F_{a,b}^\chi (c,n)$ to describe the following general class of  probability densities 
\begin {align}
\label{Fchidef}
\frac{\chi^2_{a}(Z)/a}{\chi^2_b/b}\sim F_{a,b}^\chi (c,n)
\end{align}
where $Z\sim c\chi^2_n$ and all the chi-squared variables are independent. 
\end{proposition}
Moreover, we have the following remark on the distribution of $F_{b_1,c_1}^\chi \left(\frac{\rho^2}{1-\rho^2},2n\right)$.

\begin{remark}
It is not difficult to show that the probability density of $X\sim F_{b_1,c_1}^\chi \left(\frac{\rho^2}{1-\rho^2},2n\right)$ is given by
\begin{align}
f_X(x)=\frac{(1-\rho^2)^n}{\mathcal{B}\left(\frac{c_1}{2},\frac{b_1}{2}\right)}
\left(\frac{c_1}{b_1}\right)^{\frac{c_1}{2}} \frac{x^{\frac{b_1}{2}-1}}{\left(x+\frac{c_1}{b_1}\right)^{\frac{1}{2}(c_1+b_1)}}
{}_2F_1\left(n,\frac{1}{2}(c_1+b_1);\frac{b_1}{2};\frac{x\rho^2}{x+\frac{c_1}{b_1}}\right)
\end{align}
where ${}_2F_1(a,b;c;z)$ is the Gauss hypergeometric function and $\mathcal{B}(p,q)$ is the beta function.
\end{remark}

%-------------------------------------------------------
\input{eigenvector}

%------------------------------------------------------------------------------------------------------

%-----------------------
\input{applications}

%-----------------------
\input{simulations}

\section*{Acknowledgments}
This work was supported in part by NIH grant NIH BIB R01EB1988 (PD) and BSF 2012-159 (PD, BN, OS).
%\newpage

\appendix

\input{appendixA}

\section{Proof of the auxiliary lemmas}
\label{sec:proofs_lemmas}

\begin{proof}[Proof of Lemma 1]
Write the \(m\times n\) matrix $X(\epsilon) = [x_1, \dots, x_n]$ and observe that $X(-\epsilon) =U X(\epsilon)$, where $U=diag(1,-1,\dots,-1)$, is an orthogonal matrix. Thus, the matrix  $H(-\epsilon)= U^T H(\epsilon) U$ has the same eigenvalues as $H(\epsilon) $. In particular,  the largest eigenvalue $\ell_1$ and its corresponding eigenvector $v_1$ satisfy
\begin{equation}
        \label{eq:L_1_even}
  \ell_1(-\epsilon) = \ell_1(\epsilon), \qquad v_1(-\epsilon) = Uv_1(\epsilon).
\end{equation}
Hence $\ell_1$ and the first component of $v_1$ are even functions of $\epsilon$ whereas the remaining components of $v_1$ are odd.
Denote the following matrices:
\begin{equation}
\label{eq:L_1_P_matrices}
 A_0 = \begin{pmatrix}z & 0 \\ 0 & 0_{m-1} \end{pmatrix}, \qquad
   A_1 = \sqrt{z}\begin{pmatrix}0 & b^\dagger \\ b & 0_{m-1} \end{pmatrix}, \qquad
   A_2 = \begin{pmatrix}0 & 0 \\ 0 & Z \end{pmatrix}.
\end{equation}

We decompose the matrix $H(\epsilon) $ as
\begin{eqnarray}
H = \sum_{j=1}^{n} x_j x_j^\dagger &=& \sum_{j=1}^{n} (u_j {\bf e}_1 + \epsilon \xi_j ^ \bot) (u_j {\bf e}_1 + \epsilon \xi_j ^ \bot)^\dagger \nonumber\\
& = &
\sum_{j=1}^{n} |u_j|^2 {\bf e}_1 {\bf e}_1^T +\epsilon
\sum_{j=1}^{n}[ \xi_j ^ \bot \cdot  \overline{u_j} {\bf e}_1^T + u_j {\bf e}_1 \cdot  {\xi_j ^ \bot}^\dagger ] +
\epsilon^{2}\sum_{j=1}^{n}  \xi_j ^ \bot \cdot {\xi_j ^ \bot}^\dagger \nonumber\\
&=&\begin{pmatrix}z & 0 \\ 0 & 0_{n-1} \end{pmatrix} + \epsilon\sqrt{z}\begin{pmatrix}0 & b^\dagger \\ b & 0_{n-1} \end{pmatrix} + \epsilon^2\begin{pmatrix}0 & 0 \\ 0 & Z \end{pmatrix}\nonumber
\end{eqnarray}
meaning
\begin{equation}
\label{eq:L_1_P_decomp_h}
H(\epsilon) = A_0 + \epsilon A_1 + \epsilon^2 A_2.
\end{equation}

Since \(H(\epsilon)\) is Hermitian for all \(\epsilon\), it follows that the largest eigenvalue \(\ell_1(\epsilon)\) is real-valued. Furthermore, since $H(\epsilon)$ is an holomorphic symmetric function of $\epsilon$, it follows from Kato (\cite{Kato}, Theorem 6.1, page 120) that the largest eigenvalue $\ell_1$ and its eigenprojection $P(\epsilon)$ are analytic functions of $\epsilon$, in some neighborhood of zero.

For $\epsilon=0$, ${\bf e}_1$ is an eigenvector with eigenvalue $z$, that is, $\langle P(0){\bf e}_1, {\bf e}_1\rangle = \langle {\bf e}_1, {\bf e}_1\rangle = 1$. Since $P$ is an analytic function of \(\epsilon\) and the inner product is a smooth function, the function $\langle P(\epsilon){\bf e}_1, {\bf e}_1\rangle $ is a (real-valued) analytic function of $\epsilon$ and also strictly positive in some neighborhood of $\epsilon=0$. We may thus define
\begin{equation}
\label{eq:L_1_v1_def}
v_1(\epsilon) \ = \langle P(\epsilon) {\bf e}_1, {\bf e}_1\rangle ^{-1}  P(\epsilon) {\bf e}_1.
\end{equation}
Clearly, $v_1$ is an eigenvector corresponding to the eigenvalue  $\ell_1$ and it is also analytic in some neighborhood of zero.\\
We may thus expand $\ell_1$ and $v_1$ in a convergent Taylor series in $\epsilon$. Eq. (\ref{eq:L_1_even}) implies that all odd coefficients vanish in the expansion for \(\ell_1 \):
$$ \ell_1 = \lambda_0 + \epsilon^2 \lambda_2 + \epsilon^4 \lambda_4 + \dots $$
$$ v_1 = w_0 + \epsilon w_1 + \epsilon^2 w_2 + \epsilon^3 w_3 + \epsilon^4 w_4 + \dots $$
Inserting this expansion into the eigenvalue equation $Hv_1 = \ell_1 v_1$ gives the following set of equations for $r \geq 0$
\begin{equation}
\label{eq:L_1_P_ev_eq}
A_0 w_r + A_1 w_{r-1} + A_2 w_{r-2} = \lambda_0 w_r + \lambda_2 w_{r-2} + \lambda_4 w_{r-4} + \dots
\end{equation}
with the convention that vectors with negative subscripts are zero. From the $r=0$ equation, $A_0 w_0 = \lambda_0 w_0$, we readily find that
\begin{align*}
 \lambda_0 = z, \qquad w_0 = const \cdot {\bf e}_1.
 \end{align*}
Using Eq. (\ref{eq:L_1_v1_def}), 
$$\langle v_1, {\bf e}_1\rangle =  \langle \langle P(\epsilon) {\bf e}_1, {\bf e}_1\rangle ^{-1}  P(\epsilon){\bf e}_1,  {\bf e}_1\rangle = \langle P(\epsilon) {\bf e}_1, {\bf e}_1\rangle ^{-1} \langle P(\epsilon){\bf e}_1,  {\bf e}_1\rangle = 1$$
and  $w_0 = v_1(0) = {\bf e}_1$. This implies that $w_j$, for $j \geq 1$, is orthogonal to ${\bf e}_1$, that is orthogonal to $w_0$.

From the eigenvector remarks following (\ref{eq:L_1_even}) it follows that $w_{2j}=0$ for $j \geq 1$. These remarks allow considerable simplification of equations (\ref{eq:L_1_P_ev_eq}); we use those for $r=1$ and $r=3$ 
\begin{equation}
\label{eq:L_1_P_ev_eq_2}
A_1w_0 = \lambda_0 w_1, \qquad A_2w_1 = \lambda_0 w_3 + \lambda_2 w_1
\end{equation}
from which we obtain, by setting $\hat b = \begin{pmatrix} 0 \\ b \end{pmatrix}$,
\begin{align*}
 w_1 = z ^ {-1/2} \hat b, \qquad w_3 = \lambda_0^{-1}(A_2 - \lambda_2I)w_1.
 \end{align*}
Multiply (\ref{eq:L_1_P_ev_eq}) on the left by $w_0 ^ H$ and use the first equation of (\ref{eq:L_1_P_ev_eq_2}) to get, for $r$ even,
$$ \lambda_r = (A_1 w_0)^\dagger w_{r-1} = \lambda_0 w_1^\dagger w_{r-1}$$
and hence
\begin{align*}
\lambda_2 &= \lambda_0 w_1^\dagger w_1 = b^\dagger b \\
\lambda_4 &= w_1^\dagger(A_2 - \lambda_2I)w_1 = z^{-1}b^\dagger(Z - b b^\dagger) b.
\end{align*}
\end{proof}

To prove lemma 2 and 3, we require the following two claims, which are the complex analogies of two theorems from Muirhead (\cite{Muirhead}, page 93-96, theorems 3.2.8 and 3.2.11). 

%----------------------------------------------------------------------------------------
\begin{claim}\label{claim:1}

Suppose $A \sim CW_m(n, \Sigma)$ where $A$ and $\Sigma$ are partitioned as follows
$$A =
\begin{pmatrix}
A_{11} & A_{12} \\
A_{21} & A_{22}
\end{pmatrix}
\qquad
\Sigma =
\begin{pmatrix}
\Sigma_{11} & \Sigma_{12} \\
\Sigma_{21} & \Sigma_{22}
\end{pmatrix}
$$
and let $A_{11\cdot 2} = A_{11} - A_{12} A_{22}^{-1}A_{21}$, and let $\Sigma_{11\cdot 2} = \Sigma_{11} - \Sigma_{12} \Sigma_{22}^{-1}\Sigma_{21}$. Then, $A_{11\cdot 2}$ is distributed as $\mathcal{CW}_k(n-m+k, \Sigma_{11\cdot 2})$ and is independent of $A_{12}, A_{21}$ and $A_{22}$.
\end{claim}

\begin{proof}
In the following change of variables $A_{11\cdot 2}=A_{11} - A_{12} A_{22}^{-1}A_{21}$, $B_{12} = A_{12}$, $B_{22} = A_{22}$, the Jacobian matrix is an upper-diagonal matrix, where the diagonal entries are all one. Hence
$$({\rm d}A) = ({\rm d}A_{11})\wedge({\rm {\rm d}A}_{12})\wedge({\rm d}A_{22}) = ({\rm d}A_{11\cdot 2})\wedge({\rm d}B_{12})\wedge({\rm d}B_{22}).$$
We note that
$$ \det A = \det A_{22} \det (A_{11} - A_{12}A_{22}^{-1}A_{21}) = \det B_{22}\det A_{11\cdot 2}$$
and
$$ \det\Sigma = \det\Sigma_{22}\det\Sigma_{11\cdot 2}.$$
Denote by $\Sigma^{-1} = \begin{pmatrix} C_{11} & C_{12} \\ C_{21} & C_{2} \end{pmatrix}$. Now,\\
$$ \text{tr}(\Sigma^{-1} A) =
\text{tr}\left[
\begin{pmatrix}
C_{11} & C_{12} \\
C_{21} & C_{22}
\end{pmatrix}
\begin{pmatrix}
A_{11\cdot 2} + B_{12}B_{22}^{-1}B_{21} & B_{12} \\
B_{21} & B_{22}
\end{pmatrix}
\right] = $$ $$
\text{tr} \left[ C_{11}A_{11\cdot 2} \right] +
\text{tr} \left[ C_{11}B_{12}B_{22}^{-1}B_{21} \right] +
\text{tr} \left[ C_{12}B_{21} \right] +
\text{tr} \left[ C_{21}B_{12} \right] +
\text{tr} \left[ C_{22}B_{22} \right].
$$
Using the above equation and the fact that $\text{tr}(XY) = \text{tr}(YX)$ and $\Sigma^{-1}, A$ are self-adjoint, we get
$$ \text{tr} \left[ C_{11} (B_{12} + C_{11}^{-1}C_{12}B_{22})B_{22}^{-1} (B_{12} + C_{11}^{-1} C_{12} B_{22})^\dagger \right] + $$
$$ \text{tr} \left[ B_{22} (C_{22} - C_{21}C_{11}^{-1}C_{12})\right] +\text{tr} \left[C_{11}A_{11 \cdot 2}\right] = $$
$$ = \text{tr} \left[ C_{11}B_{12}B_{22}^{-1}B_{21} \right] + \text{tr} \left[ C_{12}B_{21} \right] +
\text{tr} \left[ C_{21}B_{12} \right] + \text{tr} \left[ C_{12}C_{11}^{-1}C_{21}B_{22} \right] +
$$$$
+\text{tr} \left[ B_{22}C_{22} \right] - \text{tr} \left[ B_{22}C_{21}C_{11}^{-1}C_{12} \right] +
\text{tr} \left[ C_{11}A_{11\cdot 2} \right] = \text{tr}(\Sigma^{-1} A).
$$
Using the relations $C_{11} = \Sigma_{11\cdot 2}^{-1}$, $C_{22}-C_{21}C_{11}^{-1}C_{12} = \Sigma_{22}^{-1}$ and $ C_{11}^{-1}C_{12} = -\Sigma_{12} \Sigma_{22}^{-1}$, we get
$$ \text{tr}(\Sigma^{-1} A) = \text{tr} \left[ \Sigma_{11\cdot 2}^{-1} (B_{12} - \Sigma_{12}\Sigma_{22}^{-1}B_{22})B_{22}^{-1} (B_{12} - \Sigma_{12} \Sigma_{22}^{-1} B_{22})^\dagger \right] + $$
$$ \text{tr} \left[ B_{22} \Sigma_{22}^{-1} \right] + \text{tr} \left[\Sigma_{11\cdot 2}^{-1} A_{11 \cdot 2}\right]. $$
Now, the we can present the joint density of $A$ \cite{Goodman_63} as
$$ p(A) = \frac{1}{\mathcal{C}\Gamma_m(n) (\det \Sigma)^n} \text{etr}(-\Sigma^{-1} A) (\det A)^{n-m} = $$
$$
\frac{\text{etr}(-\Sigma_{11\cdot 2}^{-1} A_{11\cdot 2}) (\det A_{11\cdot 2})^{n-m+k-k}}{\mathcal{C}\Gamma_{k}(n-m+k)\cdot(\det \Sigma_{11\cdot 2})^{n-m+k}} \cdot
\frac{\text{etr}(-\Sigma_{22}^{-1} B_{22}) (\det B_{22})^{n-m+k}}{\mathcal{C}\Gamma_{m-k}(n)(\det \Sigma_{22})^{n}} \cdot $$$$
\cdot \frac{\mathcal{C}\Gamma_{k}(n-m+k)\cdot \mathcal{C}\Gamma_{m-k}(n) \cdot \text{etr}(-\Sigma_{11\cdot 2}^{-1} (B_{12} - \Sigma_{12}\Sigma_{22}^{-1}B_{22})B_{22}^{-1} (B_{12} - \Sigma_{12} \Sigma_{22}^{-1} B_{22})^\dagger)
 }{\mathcal{C}\Gamma_m(n)\cdot (\det \Sigma_{11\cdot 2})^{m-k}\cdot (\det B_{22})^{k}} \cdot
$$
where $\mathcal{C}\Gamma_a(b) = \pi^{\frac{1}{2}a(a-1)}\Gamma(b)\cdots\Gamma(b-a+1)$.\\
By this decomposition of the density function, and by the note on the change of variables, we conclude that $A_{11\cdot 2}$ is distributed as $\mathcal{CW}_{k}(n-m+k, \Sigma_{11\cdot 2} )$ and is independent of $B_{12}, B_{21}$ and $B_{22}$, and also of $A_{12}, A_{21}$ and $A_{22}$.
\end{proof}

%----------------------------------------------------------------------------------------
\begin{claim}\label{claim:2}
Let $A\sim\mathcal{CW}_m(n, \Sigma)$ and let $M$ be a $k \times m$ matrix of rank $k$, where $M$ is independent of $A$. Then $(MA^{-1}M^\dagger)^{-1}\sim\mathcal{CW}_k(n-m+k, (M \Sigma^{-1} M^\dagger)^{-1})$.
\end{claim}

\begin{proof}
Set $B = \Sigma^{-1/2}A\Sigma^{-1/2}$. Now $B\sim\mathcal{CW}_m(n, I)$. For $R=M \Sigma^{-1/2}$,
$$ (MA^{-1}M^\dagger)^{-1}  = (R B^{-1} R^\dagger)^{-1} $$
$$ (M\Sigma^{-1}M^\dagger)^{-1} = (R R^\dagger)^{-1}$$
thus it is sufficient to prove that $(R B^{-1} R^\dagger)^{-1}\sim\mathcal{CW}_k(n-m+k, (R R^\dagger)^{-1})$. Let $R = L[I_k:0]H$ be the SVD decomposition of $R$, where $L$ is $k \times k$ and nonsingular and $H$ is $m \times m$ unitary. Now, 
$$ (R B^{-1} R^\dagger)^{-1} = \left( L[I_k:0]H B^{-1} H^\dagger[I_k:0]' L^\dagger \right)^{-1} = $$
$$ = (L^{-1})^\dagger\left( [I_k:0] (HBH^\dagger)^{-1} [I_k:0]' ] \right)^{-1} L^{-1} = $$
$$ = (L^{-1})^\dagger\left( [I_k:0] C^{-1} [I_k:0]' ] \right)^{-1} L^{-1} $$
where $C=HBH^\dagger\sim \mathcal{CW}_m(n, I)$. Let
$$
D = C^{-1} =
\begin{pmatrix}
D_{11} & D_{12} \\
D_{21} & D_{22}
\end{pmatrix}, \quad
C =
\begin{pmatrix}
C_{11} & C_{12} \\
C_{21} & C_{22}
\end{pmatrix}
$$
where $D_{11}$ and $C_{11}$ are $k \times k$. Then $(R B^{-1} R^\dagger)^{-1} = (L^{-1})^\dagger D_{11}^{-1}L^{-1}$, and since $D_{11}^{-1} = C_{11}-C_{12}C_{22}^{-1}C_{21}$, it follows from Claim \ref{claim:1} that $D_{11}^{-1}\sim\mathcal{CW}_k(n-m+k, I_k)$. Hence $(L^{-1})^\dagger D_{11}^{-1}L^{-1}\sim\mathcal{CW}_k(n-m+k, (LL^\dagger)^{-1})$, and since $(LL^\dagger)^{-1} = (RR^\dagger)^{-1}$, the proof is complete.\\
\end{proof}

\begin{proof}[Proof of Lemma 2]
Note that $S^{11} = E^{11} = {\bf e}_1^T E^{-1} {\bf e}_1$. Then, according to Claim \ref{claim:2}, $(S^{11}) ^{-1} \sim \mathcal{CW}_1(n-m+1, I_1) = \frac{\chi_{2n-2m+2}^{2}}{2}$, meaning $S^{11} \sim \frac{2}{\chi_{2n-2m+2}^{2}}$. \\
Next, by definition $S = (M^\dagger E^{-1}M)^{-1}$, with fixed $M$. Therefore, according to  the same claim, $S \sim \mathcal{CW}_2(n-m+2, D)$, where $D = ([{\bf e}_1,  b] [{\bf e}_1,  b]^\dagger)^{-1} = diag(1, \frac{1}{||b||^2})$. Therefore $S_{22} \sim \frac{\chi_{2n-2m+4}^{2}}{2||b||^2}$. \\
Finally, notice that $(S^{11})^{-1} = S_{11} - S_{12}S_{22}^{-1}S_{21}$, thus according to Claim \ref{claim:1}, $(S^{11})^{-1}$
is independent of $S_{22}$, meaning $S^{11}$ is independent of $S_{22}$.

\end{proof}

\begin{proof}[Proof of Lemma 3]
Since $A_2$ is independent of $E$, we have
$$
\mathbb{E}[A_2|E]=\mathbb{E}[A_2] =
\begin{pmatrix}
0 & 0 \\
0 & I_{m-1}.
\end{pmatrix}
$$
Hence$$
\mathbb{E}\left[\frac{ {\bf e}_1^T E^{-1} A_2 E^{-1} {\bf e}_1 }{E^{11}}\right]=
\mathbb{E}_E\left[\mathbb{E}_{A|E}\left[\frac{ {\bf e}_1^T E^{-1} A_2 E^{-1} {\bf e}_1 }{E^{11}}\right]\right]= \mathbb{E}\left[\sum_{j=2}^{m}\frac{||E^{1j}||^2}{E^{11}}\right] =
(m-1) \mathbb{E}\left[\frac{||E^{12}||^2}{E^{11}}\right].
$$
To compute this expectation, consider the matrix $S^{-1} = [{\bf e}_1 \; e_2] ^T E^{-1} [{\bf e}_1 \; e_2] = \begin{pmatrix} E^{11} & E^{21} \\ \overline{E}^{21} & E^{22}\end{pmatrix}$. Now, $S^{22}= E^{22}$ and $S_{22} = E^{11}/(E^{11}E^{22} - ||E^{12}||^2)$. Hence,
\begin{equation}
\frac{1}{S_{22}} = E^{22} - \frac{||E^{12}||^2}{E^{11}}.
\end{equation}
We now take expectation of the above equation. As in the previous lemma, $S_{22} \sim \frac{1}{2} \chi_{2n - 2m + 4}^2$, whereas $E^{22} \sim \frac{2}{\chi_{2n-2m+2}^{2}}$. Using these results, we obtain
\begin{equation}
\mathbb{E}\left[ \frac{(E^{12})^2}{E^{11}}\right]  =\mathbb{E}[E^{22}]  - \mathbb{E}\left[\frac{2}{\chi^2_{2n-2m+4}}\right] = \frac{2}{2n-2m} - \frac{2}{2n-2m+2}=\frac{1}{(n-m)(n-m+1)}
\end{equation}
which completes the proof.
\end{proof}

%\bibliographystyle{unsrt}
%\bibliography{bibfile}

\end{document}

%% file: abstract.tex
\begin{abstract}\small
The largest eigenvalue of a Wishart matrix, known as Roy's largest root (RLR), 
plays an important role in a variety of applications. Most works to date 
derived approximations to its distribution under various asymptotic regimes, such as degrees of freedom, 
dimension, or both tending to infinity. However, several applications involve finite and relative small 
parameters, for which the above approximations may be inaccurate.
Recently, via a small noise perturbation approach with fixed dimension and degrees of freedom, Johnstone 
and Nadler derived simple yet accurate stochastic approximations to the distribution of Roy's largest root in the real 
valued case, under a rank-one alternative.
In this paper, we extend their results to the complex valued case. Furthermore, we analyze the behavior of the leading eigenvector by developing new stochastic approximations. 
Specifically, we derive simple stochastic approximations to the distribution of the largest eigenvalue under five common complex
single-matrix and double-matrix scenarios. We then apply these results to investigate several problems in 
signal detection and communications. In particular, we analyze the performance of RLR detector in cognitive radio spectrum sensing and constant-modulus signal detection in the high signal-to-noise ratio (SNR) regime. Moreover, we address the problem of determining
the optimal transmit-receive antenna configuration (here optimality is in the sense of outage minimization) for rank-one multiple-input and multiple-output Rician-Fading channels at high SNR.
\end{abstract}
%-Abstract

%% file: intro.tex
%----------------------------------------------------------------------------------------
\section{Introduction}
%----------------------------------------------------------------------------------------

Let $H, E$ be two independent complex-valued Wishart matrices, where $E \sim \mathcal{CW}_m(n_E, \Sigma_E)$ and $H$ is either central or non-central Wishart, namely $H \sim \mathcal{CW}_m(n_H, \Sigma_H)$ or $H\sim\mathcal{CW}_m(n_H, \Sigma_H, \Omega)$, respectively. We denote the largest eigenvalue of \(H\) by $\ell_1(H)$ and similarly, the largest eigenvalue of $E^{-1}H$ by $\ell_1(E^{-1}H)$. These largest eigenvalues, either in the single matrix case or in the double matrix case, are central quantities of interest in many applications, specifically in signal detection and communications. More generally, these eigenvalues, also known as Roy's largest roots, play a key role in hypothesis testing problems.

Obtaining simple expressions, exact or approximate, for the distribution of $\ell_1$ in the single or double matrix case has been a subject of intense research for over more than 50 years.
An exact expression for the distribution of  $\ell_1$, in the single central matrix case with an identity covariance matrix ($\Sigma_H = I$), was first presented by Khatri \cite{Khatri_64}. This result was generalized to various other settings, such as an arbitrary covariance matrix or a non-centrality matrix \cite{Khatri_69, Kang_2003, Ratnarajah_2005}.
The resulting expressions are, in general,
 challenging to evaluate numerically.
More recently, Zanella {\em et al}. \cite{Zanella_Chiani} presented simpler exact expressions, both for the central case with arbitrary \(\Sigma_H\), as well as the noncentral case but with $\Sigma_H=I$, that are easier to evaluate, though still require a recursive algorithm.

A different approach to derive approximate distributions for the largest eigenvalue in the null case, where $\Sigma_{E}=\Sigma_{H}=I$, is based on random matrix theory. Considering the limit as $n_H$ and $m$ (and in the double matrix case also \(n_{E}\)) tend to infinity, with their ratios kept fixed, $\ell_1$ in the single matrix case, and $\log(\ell_1)$ in the double matrix case, asymptotically follow  a Tracy-Widom distribution \cite{Johansson, Johnstone_2001, Johnstone_2009}. Furthermore, with suitable centering and scaling coefficients, the convergence to these limiting distributions can be quite fast \cite{Za_2012, Karoui_2006}.

In this paper we focus on the distribution of \(\ell_1\) under a rank-one alternative with complex valued observations, namely when $\Sigma_{H} = I + \lambda {\bf v}{\bf v}^\dagger$ in the central case, or $\Omega= \lambda {\bf v}{\bf v}^\dagger$ in the non-central case, where  \({\bf v}^{\dagger}\) denotes the conjugate transpose of $\bf{v}$.
One classical result in the single-matrix case, is that asymptotically as $n_{H}\to\infty$ with fixed dimension $m$, $\ell_1(H)$ asymptotically follows a Gaussian distribution \cite{Anderson}. Recently, Paul \cite{Paul_2007} proved that in the random matrix setting, as both $n_H$ and $m$ tend to infinity with their ratio fixed, if \(\lambda>\sqrt{m/n_H}\) then $\ell_1(H)$ still converges to a Gaussian distribution.
In the double-matrix case, the location of the phase transition and the limiting value of \(\ell_1(E^{-1}H)\) were recently studied by Nadakuditi and Silverstein \cite{Nadakuditi_Silverstein_2010}. Moreover, in a very recent development, the authors in \cite{Dharmawansa} have proved that, above the phase transition, $\ell_1(E^{-1}H)$ converges to a Gaussian distribution.

%While it seems reasonable that in analogy to the single-matrix case, asymptotically \(\ell_1(E^{-1}H)\) also follows a Gaussian distribution when above the phase transition, to the best of our knowledge, such a distributional result has not yet been proven.

Whereas the above results assume that dimension and degrees of freedom tend to infinity, in various common applications these quantities are not only finite but typically relatively small. In such settings, the above asymptotic results may yield poor approximations to the distribution of the largest eigenvalue \(\ell_1\), which may be quite far from 
Gaussian (see Figure \ref{fig:prop1_2} for an illustrative example).
Recently, using a small noise perturbation approach, Johnstone and Nadler \cite{Johnstone_Nadler} derived approximations to the distribution of  $\ell_1$, both for single and double real-valued Wishart matrices with finite dimension and degrees of freedom. In this paper, we build upon their work and extend their results to the complex valued case. Propositions 1-5 of Section \ref{sec:propositions} provide approximate expressions for the distribution of $\ell_1$ corresponding to the five common single-matrix and double-matrix cases outlined in Table \ref{table:4_cases}.
Furthermore, in section \ref{sec:inner} we study the fluctuations in the leading \textit{eigenvector}, in particular its overlap with the population eigenvector.

Next, in section \ref{sec:application}, we illustrate the utility of these approximations in several applications in signal detection and communication. For signal detection, we use propositions 1-4 to provide simple approximate expressions for the power of Roy's largest root test under two common signal models. Next, we consider the outage probability in a multiple-input and multiple-output (MIMO) communication system. For the particular case of a rank-one Rician fading channel, we use Proposition 2, and show \textit{analytically} that to minimize the outage probability it is preferable to have equal number of transmitting and receiving antennas. This important design property 
was previously observed via simulations \cite{Kang_2003}. Finally, Section \ref{sec:simulations} provides some simulation results to verify the accuracy of the new results.

\begin{table}[t]
\begin{center}
    \begin{tabular}{| c |  p{5cm} | p{9cm} | }
    \hline
    Case & Distribution & Testing Problem, Application  \\ \hline

    1 & $H \sim \mathcal{CW}_m(n_H, \Sigma+\lambda\bf{v}\bf{v}^\dagger)$  & Signal detection in noise,  \\
   & $\Sigma$ is known & known noise covariance matrix. \\ \hline

    2 & $H \sim \mathcal{CW}_m(n_{H}, \Sigma,\frac{\omega}{\det(\Sigma)}\bf{v}\bf{v}^\dagger)$ $\Sigma$ is known &  Constant modulus signal detection in noise, known noise covariance matrix. \newline Outage probability of a Rician-fading MIMO channel. \\ \hline

    3 & $H \sim \mathcal{CW}_m(n_{H}, \Sigma+\lambda\bf{v}\bf{v}^\dagger)$  & Signal detection in noise, \\
 & $E\sim \mathcal{CW}_m(n_{E}, \Sigma)$ & estimated noise covariance matrix.  \\ \hline

    4 & $H \sim \mathcal{CW}_m(n_{H}, \Sigma,\frac{\omega}{\det(\Sigma)}\bf{v}\bf{v}^\dagger)$ &  Constant modulus signal detection in noise, \\
& $E\sim \mathcal{CW}_m(n_{E}, \Sigma)$  & estimated noise covariance matrix. \\ \hline

5 & $H\sim \mathcal{CW}_p(q,\Phi,\Omega)$ & canonical correlation analysis\\
& $E\sim \mathcal{CW}_p(n-q,\Phi)$, $\Omega$ random & between two groups of sizes $p\leq q$\\
\hline

    \end{tabular}
\end{center}
\caption{Five common single-matrix and double-matrix settings  and some representative applications.}
\label{table:4_cases}
\end{table}

%% file: eigenvector.tex
%----------------------------------------------------------------------------------------------------
\section{On the inner product between the sample and population eigenvectors}\label{sec:inner}

We now consider the relation between the leading sample eigenvector and the population eigenvector. In most practical scenarios, knowledge of the exact population covariance matrix is not available and in particular $\mathbf{v}$ is unknown. Therefore, it is common to use the sample eigenvalues/eigenvectors instead of the population analogs. In such situations, key quantity of interest is the correlation between the leading sample eigenvector and its population counterpart. This correlation can be represented as follows
\begin{align}
\label{R}
R&=\frac{|\hat{\mathbf{v}}^\dagger \mathbf{v}|^2}{||\hat{\mathbf{v}}||^2}
\end{align}
where $\hat{\mathbf{v}}$ denotes the sample eigenvector corresponding to the largest eigenvalue \(\ell_1\). 

This quantity is of considerable interest both theoretically and practically. Theoretically, the large sample almost sure limits in the random matrix setting (i.e., as $m,n_H\to\infty$ such that $m/n_H\to \kappa>0$) were derived in  \cite{Paul_2007} and  \cite{nadler2008finite}. These results show that $R$ does not converge to $1$ (i.e.,$\hat{\mathbf{v}}$ is inconsistent).
For a practical application, the quantity $R$ is of paramount importance in the design of adaptive beamformers (ABFs) in array processing. In this respect, the dominant mode rejection (DMR) ABF derives its weight vector using the eigen-decomposition of the sample covariance matrix \cite{Van_Trees}. Since one of the main purposes of the beamformers is to eliminate loud interferers coming from undesired directions other than the steering direction, the DMR ABF creates deep notches along the directions of loud interferers. As shown in \cite{wage2014snapshot}, an important parameter which determines the  depth of attenuation of the interferers is the correlation between the sample eigenvectors and unobservable population eigenvectors. Moreover, in the presence of a single dominant interferer, the population covariance matrix takes the form of a rank one spiked model (see \cite[Eq. 17]{wage2014snapshot}). Therefore, it is important to understand the behavior of $R$ under the rank one spiked framework. 

Motivated by the above facts, in what follows we develop stochastic approximations to $R$ in two scenarios depending on  $\hat{\mathbf{v}}$. Specifically, in the first scenario, $\hat{\mathbf{v}}$ is the eigenvector corresponding to the largest eigenvalue of $\mathcal{CW}_m(n,\lambda \mathbf{vv}^\dagger +\sigma^2 I_m)$ while in the second scenario it is the eigenvector corresponding to the largest eigenvalue of $\mathcal{CW}_m(n,\sigma^2 I_m,(\omega/\sigma^2) \mathbf{vv}^\dagger)$.

Now we have the following proposition in the first scenario.
\begin{proposition}\label{prop:innerv}
Let $H\sim \mathcal{CW}_m(n,\lambda \mathbf{vv}^\dagger+\sigma^2 I_m)$, with $||\mathbf{v}||=1$ and $\lambda>0$. Let $\hat{\mathbf{v}}$ be the eigenvector corresponding to the largest eigenvalue of $H$. Then, with $(m,n,\lambda)$ fixed, as $\sigma\to 0$
\begin{align}
R\approx \frac{1}{\displaystyle 1+\frac{\sigma^2}{\lambda+\sigma^2} \frac{A}{B}+\frac{2\sigma^4}{\left(\lambda+\sigma^2\right)^2} \frac{AC}{B^2}}
\end{align}
where $A\sim \chi^2_{2m-2}$, $B\sim \chi^2_{2n}$, and $C\sim \chi^2_{2n-2}$ are independent random variables.
\end{proposition}

The distribution of $R$ corresponding to the second scenario is given by the following proposition.

\begin{proposition}\label{prop:innervnon}
Let $H\sim \mathcal{CW}_m(n,\sigma^2I_m,(\omega/\sigma^2)\mathbf{vv}^\dagger)$, with $||\mathbf{v}||=1$. Let $\hat{\mathbf{v}}$ be the eigenvector corresponding to the largest eigenvalue of $H$. Then, with $(m,n,\omega)$ fixed, as $\sigma\to 0$
\begin{align}
R\approx \frac{1}{\displaystyle 1+ \frac{A}{B_\sigma}+ 2\frac{AC}{B_\sigma^2}}
\end{align}
where $A\sim \chi^2_{2m-2}$, $B_\sigma\sim \chi^2_{2n}(2\omega/\sigma^2)$, and $C\sim \chi^2_{2n-2}$ are independent random variables.
\end{proposition}

%%%%%%%%%%%%%%%%%%%%%%%%%%%%%%%%%%%%%%%%%%%%%%%%%%%%%%%%%%%%%%%%%%%%%%%%%%%%%%%%%%

%% file: applications.tex
%----------------------------------------------------------------------------------------
\section{Applications}
After establishing approximations to the distribution of Roy's largest root, we now demonstrate their utility in three different engineering applications. The first two applications are concerned with common problems in signal detection, whereas the third is concerned with the outage probability of a rank-one Rician fading MIMO channel.
\label{sec:application}
%----------------------------------------------------------------------------------------
\subsection{Signal Detection in Spectrum Sensing}\label{sec:spectrum_sensing}
In standard multiuser communication systems, each primary user (PU) is allocated a unique frequency band of the spectrum. By design, this band may be used solely by the corresponding PU. To better utilize the available spectrum, novel cognitive radio dynamic spectrum allocation methods have been proposed in the past decade \cite{Haykin05}. In these schemes, opportunistic secondary users (SU) are allowed to use frequency bands not allocated to them by first sensing whether these are currently in use by their PU's. Several measurement schemes and statistical tests were derived for this task, 
see \cite{zeng2010review},\cite{Gazor} and references therein. 

 One of the proposed test statistics, in particular when the noise level is known and the signal is assumed Gaussian, is simply the largest eigenvalue of the observed data, namely Roy's Largest Root Test (RLRT). Assuming the noise variance is small (i.e., in the high signal-to-noise ratio (SNR) regime), one may then 
easily approximate the power of this test using Propositions 1 and 3. 
  
Let us describe this setting in more detail. Consider a spectrum sensing system with $m$ receiving antennas, that in a short time window samples $n_H$ vectors  $\{ {\bf y}_j \}_{j=1}^{n_H}$, where ${\bf y}_j\in \mathbb{C}^m$.
A common modeling assumption is that the samples are i.i.d. realizations of a random vector taking the form
\begin{equation}
\label{eq:SP_signal_def}
{\bf y} = \sqrt{\lambda}s\textbf{u}+\sigma \textbf{n}
\end{equation}
where $s\sim \mathcal{CN}(0, 1)$, $\textbf{u}\in\mathbb{C}^m$ is  the (normalized)\ channel gain vector between the PU and the antennas, $||{\bf u}||=1$, $\textbf{n}$ is an additive Gaussian noise, $\textbf{n}\sim\mathcal{CN}(0, \Sigma)$, and $\lambda$ is the received signal power. If the PU is inactive, namely  $\lambda = 0$, then all measurements are just noise. Hence, the spectrum sensing task can be formulated as the following hypothesis testing problem
\begin{equation}
\label{eq:SP_hyp_1}
\begin{array}{c l}
\mathcal{H}_0: \lambda = 0 \quad &   \text{PU is inactive} \\
\mathcal{H}_1: \lambda > 0\quad&  \text{PU is active} .
\end{array}
\end{equation}

When $\Sigma$ is known (w.l.o.g. it is assumed to be of the form \(\Sigma=\sigma^2I_m\),) (\ref{eq:SP_hyp_1}) yields the following hypothesis testing problem
\begin{equation}
\begin{array}{c l}
\mathcal{H}_0: &   {\bf y}_j\sim \mathcal{CN}(0, \sigma^{2} I_m) \\
\mathcal{H}_1: &  {\bf y}_j\sim \mathcal{CN}(0, \lambda \textbf{u}\textbf{u}^\dagger +\sigma ^{2}I_m) .
\end{array}
\end{equation}
Under the alternative $\mathcal{H}_1$, the unnormalized sample covariance matrix of the $n_H$ observations follows a complex Wishart distribution,
\begin{equation}
H = YY^\dagger \sim \mathcal{CW}_m(n_H, \lambda \textbf{u}\textbf{u}^\dagger +\sigma ^{2}I_m).
\end{equation}
We may then use Proposition \ref{prop:1} to approximate the power of Roy's largest root test, for a given threshold parameter $\mu$, as
\begin{equation}
P_{D}=\Pr\left[\ell_1(\sigma, \lambda) > \mu|\mathcal{H}_1\right].
\end{equation}
Figure \ref{fig:SS_snr} demonstrates the accuracy of our approximations for several values of signal-to-noise ratio (SNR), where $\mbox{SNR}=\frac{\lambda}{\sigma^2}$.  

Another possible scenario is when the noise covariance matrix $\Sigma$ is arbitrary and unknown but we have $n_E>m+1$ i.i.d. noise only samples $\{{\bf z}_j\}$, generated from ${\bf z} = \textbf{n}$, see \cite{zhao1986detection,Nadakuditi_Silverstein_2010}. Then we can approximate $\Sigma$ by $\frac{1}{n_E}\sum_{j=1}^{n_E}{\bf z}_j{\bf z}_j^\dagger$ and whiten the sample covariance matrix.  Therefore, we get
\begin{equation}
\left(\sum_{j=1}^{n_E}{\bf z}_j{\bf z}_j^\dagger\right)^{-1}\left(\sum_{j=1}^{n_H}{\bf y}_j {\bf y}_j^\dagger\right) = E^{-1}H
\end{equation}
where
\begin{equation}
E\sim \mathcal{CW}_m(n_E, I) \qquad H\sim \mathcal{CW}_m\left(n_H, \frac{\lambda}{\sigma^2} \textbf{u}\textbf{u}^\dagger + I\right).
\end{equation}
Now we can use Proposition 3 to approximate the power of the test
\begin{equation}
P_{D}=\Pr\left[\ell_1\left(\frac{\lambda}{\sigma^2}\right) > \mu|\mathcal{H}_1\right]
\end{equation}
for a given threshold parameter $\mu$.

Unlike previous analyses which heavily relied on the assumption of large number of samples or antennas, the approximations here enable us to analyze scenarios with small and fixed   $n_E, n_H$ and $m$. However, it is noteworthy that our approximations are very tight only for small values of $\sigma$ (i.e., in the high SNR regime). 
Above facts are further illustrated in  Fig. \ref{fig:SS_snr}.

\begin{figure}[t]
\centering
\includegraphics[width=2.85in]{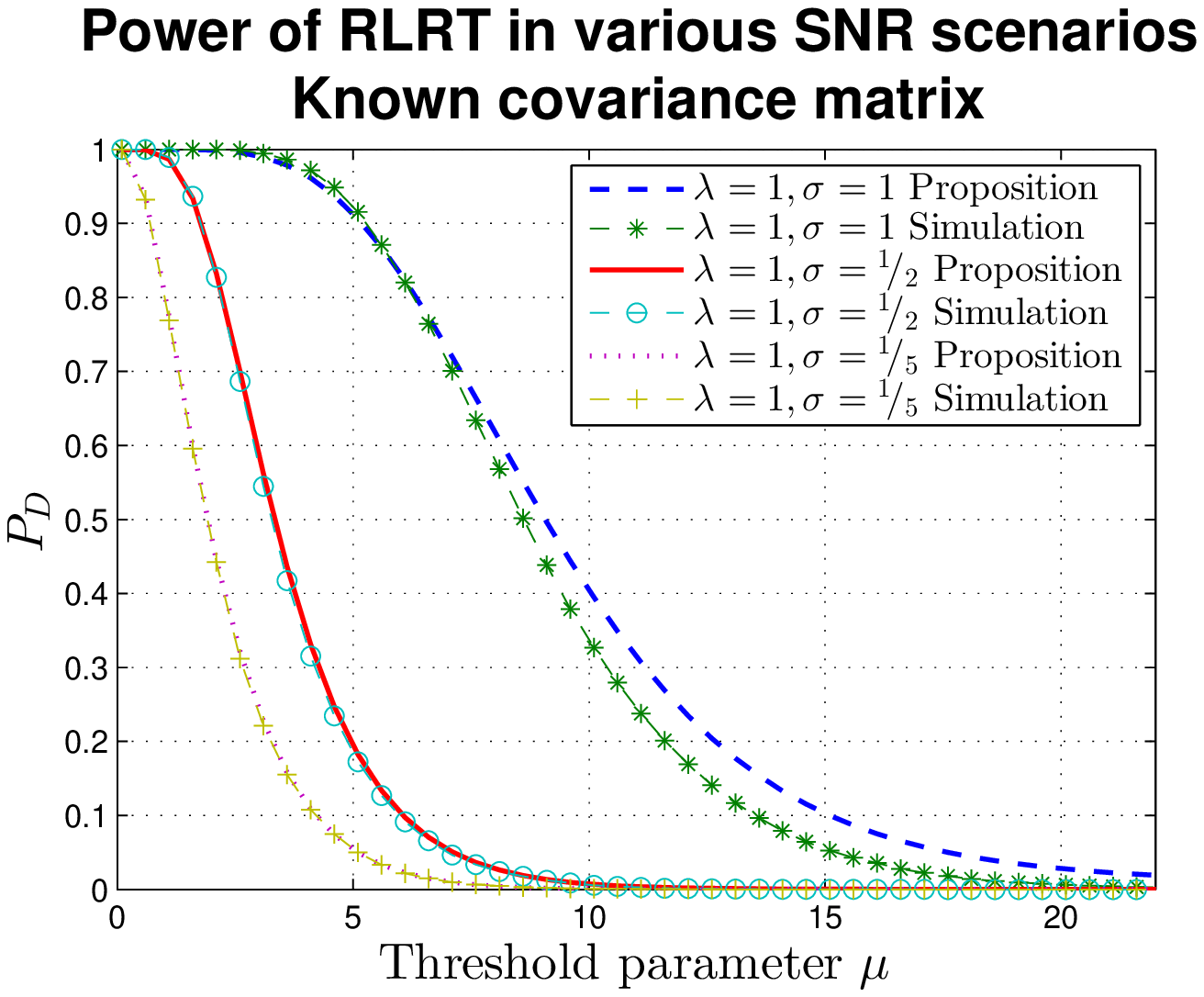}
\includegraphics[width=2.85in]{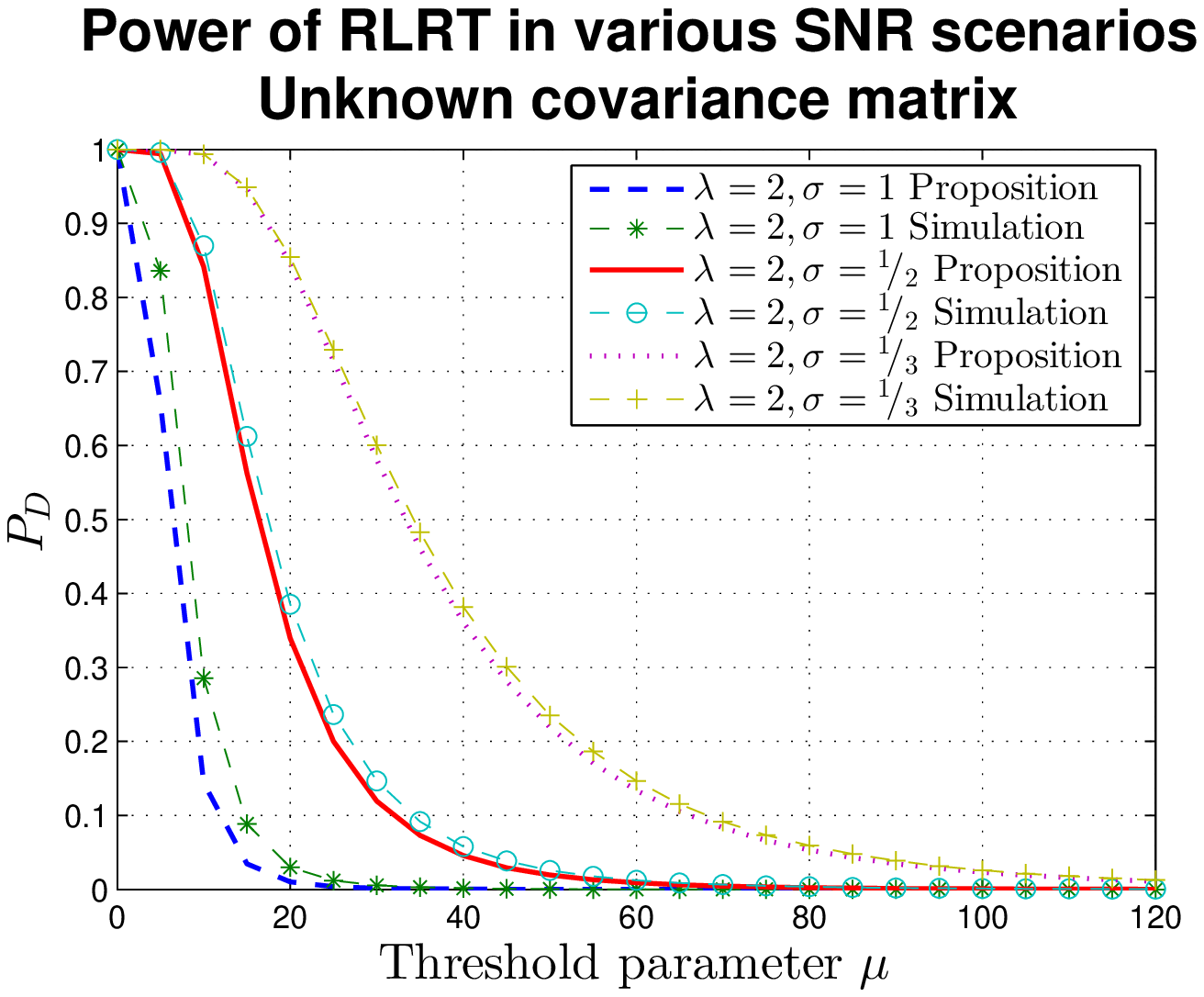}
\caption{Evaluation of RLRT power, by simulations and by our proposed approximations, for various values of SNR, as a function of the threshold \(\mu\). Clearly, as the SNR grows, the accuracy of the approximation increases. }
\label{fig:SS_snr}
\end{figure}

\subsection{Constant Modulus Signal Detection}
In the previous section, it was assumed that the transmitted signal $s$ 
was Gaussian distributed. We now consider a different setting of constant modulus (CM) signals. Here the signal \(s=e^{{\imath}\phi}\), where $\imath=\sqrt{-1}$, and $\phi\in \mathbb{R}$ is an  unknown (possibly random) time-dependent phase. One common example is the well-known FM signal \cite{Sensor_Array_Processing}. As in section \ref{sec:spectrum_sensing}, given $n_H$ measurements, the task is to decide whether they contain a CM signal, or just noise. 

Formally, we assume the availability of $n_{H}$ i.i.d. samples $\{\textbf{y}_j\}_{j=1}^{n_{H}}$  drawn from  $$\textbf{y}=\sqrt{\lambda}\exp(\imath\phi)\textbf{u}+\textbf{n}$$
where $\mathbf{n}\sim \mathcal{CN}(0,\Sigma)$ and $\lambda>0$ is the transmit signal power. The preceding detection problem can be formulated as the following hypothesis testing problem
\begin{equation}
\label{eq:CM_hyp_1}
\begin{array}{c l}
\mathcal{H}_0: \lambda = 0 \quad &   \text{signal is absent} \\
\mathcal{H}_1: \lambda > 0\quad&  \text{signal is present} .
\end{array}
\end{equation}
Although Roy's largest root test is not necessarily the optimal detector for the CM\ signals, due to its simplicity and low computational complexity, it is still a common choice. 
As in section \ref{sec:spectrum_sensing}, we may approximate
its power using our propositions. 

We assume w.l.o.g. that $\Sigma = \sigma^2I$. Then, conditional on the \(n_{H}\) phases $\phi_1,\ldots,\phi_{n_H}$, the sample covariance matrix
$H = \sum_{j=1}^{n_H}\textbf{y}_j{\bf y}_j^\dagger$ follows a non-central Wishart distribution, 
\[
H \sim \mathcal{CW}_m\left(n_H, \sigma^2I_m, \frac{\lambda n_H}{\sigma^2}{\bf vv^\dagger}\right).
\]
Importantly,  this distribution does \textit{not} depend on the phases $\phi_{j}$, If the noise level $\sigma$ is known, then we may use Proposition \ref{prop:2}, with $\omega = \lambda n_H,$ to obtain an approximation to the power of Roy's largest root test for a given threshold parameter $\mu$
\begin{equation}
P_{D}=\Pr\left[\ell_1(\sigma, \omega) > \mu|\mathcal{H}_1\right].
\end{equation} 

%SIMULATIONS AND MORE DISCUSSION - IN PROGRESS 

If the noise covariance matrix $\Sigma$ is arbitrary and unknown but we have $n_E>m+1$ i.i.d. noise only samples $\{{\bf z}_j\}$, generated from ${\bf z} = \textbf{n}$, we can approximate $\Sigma$ with $\frac{1}{n_E}\sum_{j=1}^{n_E}{\bf z}_j{\bf z}_j^\dagger$ and whiten the sample covariance  matrix. As before, we assume w.l.o.g. that $\Sigma = \sigma^2I_m$, which leads to\begin{equation}
\left(\sum_{j=1}^{n_E}{\bf z}_j{\bf z}_j^\dagger\right)^{-1}\left(\sum_{j=1}^{n_H}{\bf y}_j {\bf y}_j^\dagger\right) = E^{-1}H
\end{equation}
where
\begin{equation}
E\sim \mathcal{CW}_m(n_E, I) \qquad H\sim \mathcal{CW}_m(n_H, I_m, \lambda n_H{\bf u u^\dagger}).
\end{equation}
In this case we use Proposition \ref{prop:4}, with $\omega = \frac{\lambda n_H}{\sigma^2},$ to obtain an approximation to the power of Roy's largest root test for a given threshold parameter $\mu$.

%SIMULATIONS - ON PROGRESS

\subsection{Rank-One Rician-Fading MIMO Channel}

The last application we present involves the outage probability of a MIMO channel. Consider a general MIMO communication channel  with $n_T$ transmitters and $n_R$ receivers. The relation between the transmitted signal ${\bf x}$ and the received signal ${\bf y}$ is assumed to be of the form
\begin{equation} {\bf y} = H{\bf x} + \mathbf{n}\ \end{equation}
where  $H$ is the channel matrix of size $n_R \times n_T$ and \(\mathbf{n}\) is a complex Gaussian noise with zero mean and covariance matrix $\sigma_n^2 I_{n_R}$.
Under a common fading model, known as Rican fading \cite{Hansen_Bolcskei}, the matrix $H$ takes the form
$$ H = \sqrt{\frac{K}{K+1}}H_1 + \sqrt{\frac{1}{K+1}}H_2 $$
where the parameter \(K\) and the two matrices $H_1,H_2$ characterize the channel. The matrix $H_1$ represents the specular (Rician) component that typically results from a direct line-of-sight
between transmitter and receiver antennas.
The matrix $H_2$ represents the scattered Rayleigh-fading component which is random.
Assuming fixed sender and receiver locations, the matrix $H_1$ is constant, whereas the random matrix $H_2$ is typically modeled as i.i.d. complex Gaussians with zero mean and variance $\sigma_H^2$.
In this setting, the parameter \(K\) (a.k.a. the Rician factor) represents the ratio of deterministic-to-scattered power of the environment. Moreover, we make the common assumption that $\text{tr}^2(HH^\dagger) = n_R \times n_T$.

The SNR $\mu$ of this channel, under the maximal ratio transmission strategy\footnote{Under this strategy, the transmitter transmits information along the leading eigenvector of $H^\dagger H$.}, is given by \cite{Kang_2003}
\begin{equation}
        \label{eq:MIMO_SNR}
 \mu = \frac{\Omega_D}{\sigma_n^2} \lambda_{\max}
\end{equation}
where $\Omega_D = ||{\bf x}||^2$ is the power of the transmitted vector and $\lambda_{\max}$ is the largest eigenvalue of $HH^\dagger$.
An important quantity, which characterizes  the channel, is the outage probability. The outage probability is defined as the probability of failing to achieve a specified minimal SNR threshold  $\mu_{\min}$ required for satisfactory reception. 
Following Eq. (\ref{eq:MIMO_SNR}), the outage probability $P_{\text{out}}$ can be written as
\begin{equation}
\label{eq:MIMO_P_out}
P_{\text{out}} = \Pr\left( \frac{\Omega_D}{\sigma_n^2} \lambda_{\max} \leq \mu_{\min}\right).
\end{equation}

One particularly interesting case is when the Ricean component $H_1$ is assumed to be of rank one,  $H_1 = {\bf u} {\bf v}^\dagger$, where ${\bf u}\in \mathbb{C}^{n_R}$, ${\bf v} \in \mathbb{C}^{n_T}$. 
In this case, an important design question is which antennas configuration minimizes Eq. (\ref{eq:MIMO_P_out}), under the constraint that the sum of the number of transmitting and receiving antennas is fixed. In this respect, via simulations, 
\cite{Kang_2003} showed that the most preferable configuration is to have equal number of transmitting and receiving antennas. Here we analytically prove this result using the main approximations of this article, under the assumption of small noise (i.e., in the high SNR regime).

\begin{claim}\label{claim:mimo}
Consider a rank-one Rician fading channel with a fixel total number of antennas, $n_T+n_R=N$. Then, for $\sigma_H \ll 1$, the setting $n_T = n_R = N/2$ for $N$ even (or say $n_T = \lfloor N/2\rfloor,n_R = \lceil N/2\rceil$ for $N$ odd) minimizes the outage probability.
\end{claim}

\begin{proof}
Notice that the $j$-th column of $H$ is distributed as
$\mathcal{CN}\left(\sqrt{ \frac{K}{K+1} } {\bf u}_j {\bf v}, \frac{1}{K+1} \sigma^2_H\right)$ and therefore $HH^\dagger$ is distributed as a non-central complex Wishart matrix
\begin{equation}
H H^\dagger \sim \mathcal{CW}_{n_R}(n_T, \alpha^2 I, \beta^2 / \alpha^2 {\bf w} {\bf w}^\dagger )
\end{equation}
where ${\bf w} = {\bf v} / ||{\bf v}||, \alpha^2 = \frac{1}{K+1} \sigma_H^2$ and $\beta^2 = \frac{K}{K+1} ||{\bf u}||^2 ||{\bf v}||^2 =\frac{K}{K+1} n_R  n_T$.

Thus, the matrix $HH^\dagger$ satisfies the conditions of Proposition 2. Hence, for fixed $(n_T, n_R, K)$, as $\sigma_H \rightarrow 0$
\begin{equation}
\label{eq:MIMO_approx}
\mu =  \frac{\Omega_D}{\sigma_n^2} \lambda_{max} =
C_1 \left(X_1 +
X_{2}+
\frac{X_{2}X_3}{X_{1}} \right) +
o(\sigma_H^4)
\end{equation}
where the three random variables $X_1, X_2, X_3$ are independent with the following distributions
$$ X_1\sim\chi_{2n_T}^2(C_2), \qquad X_2\sim\chi_{2n_R-2}^2, \qquad X_3\sim\chi_{2n_T-2}^2$$ with
\begin{equation}
C_1 = \frac{\Omega_D \sigma^2_H}{2 (K+1) \sigma_n^2}, \qquad C_2 = \frac{2\beta^2}{ \alpha^2} = \frac{2 n_R  n_T K}{\sigma_H^2}.
        \label{eq:C1C2}
\end{equation}

Since Eq. (\ref{eq:MIMO_approx}) is accurate for small values of $\sigma_H$, we conclude that $C_{2}\gg1$. Therefore, we may neglect the third term and the remainder terms in  Eq. (\ref{eq:MIMO_approx}) to obtain
\begin{equation}
\label{eq:MIMO_approx2}
\mu \approx
C_1 \left(X_1 +
X_{2} \right) =C_1(\chi_{2n_T}^2(C_2)+\chi_{2n_R-2}^2).
\end{equation}
 Since $X_1$ and $X_2$ are independent, we conclude that $X_1+X_2\sim\chi_{2n_T+2n_R-2}^2(C_2)$. Thus,  
 \begin{align*}
 \mu \propto \chi^2_{2N-2}(C_2).
 \end{align*}
Clearly $P_{out}$ is minimal when the non-centrality parameter $C_2$ is maximal. Since by Eq. (\ref{eq:C1C2}) the parameter $C_2 \propto n_T \cdot n_R$, the claim follows.
\end{proof}

\begin{figure}[t]
  \centering
  \includegraphics[width=2.8in]{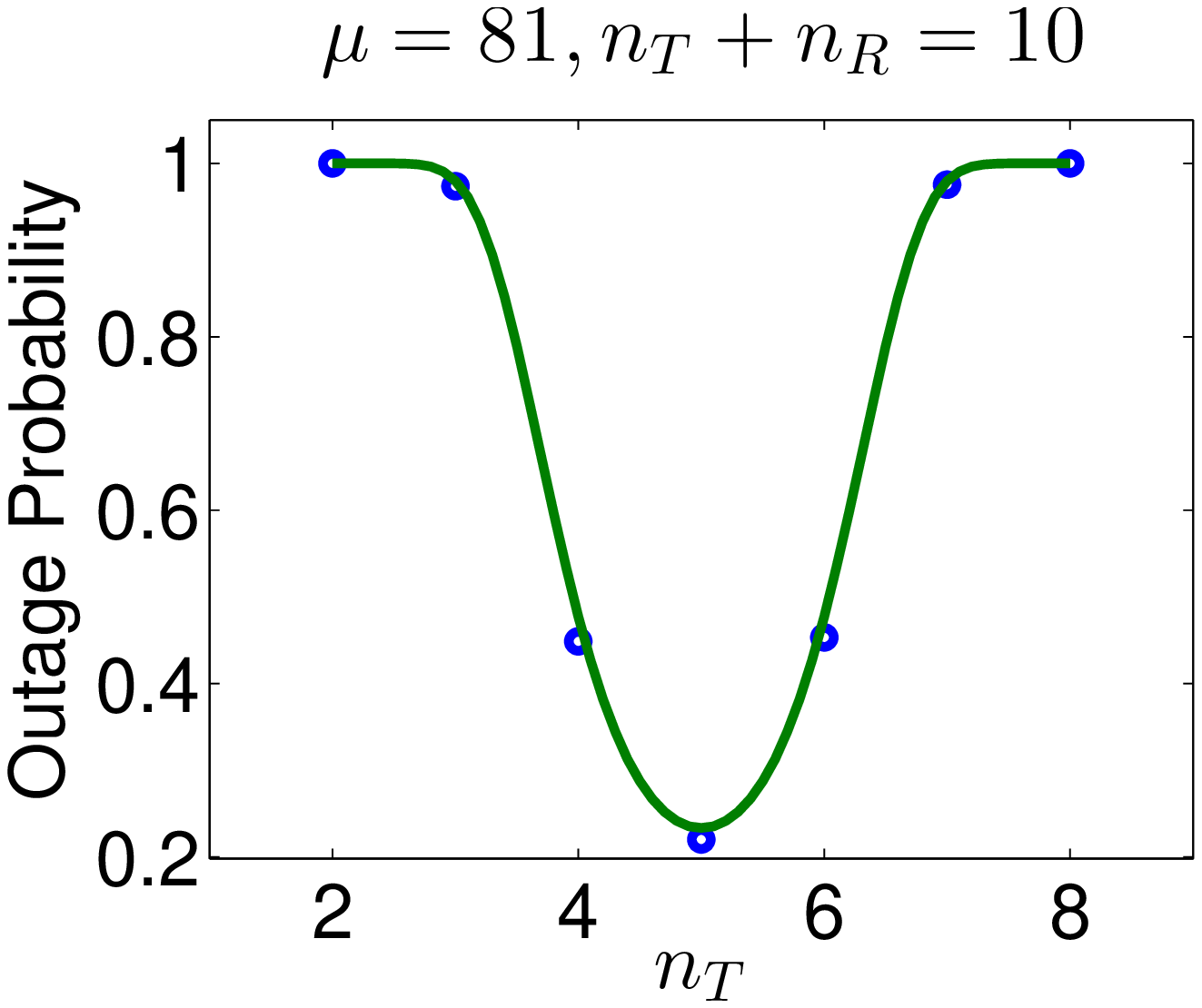}
  \includegraphics[width=2.8in]{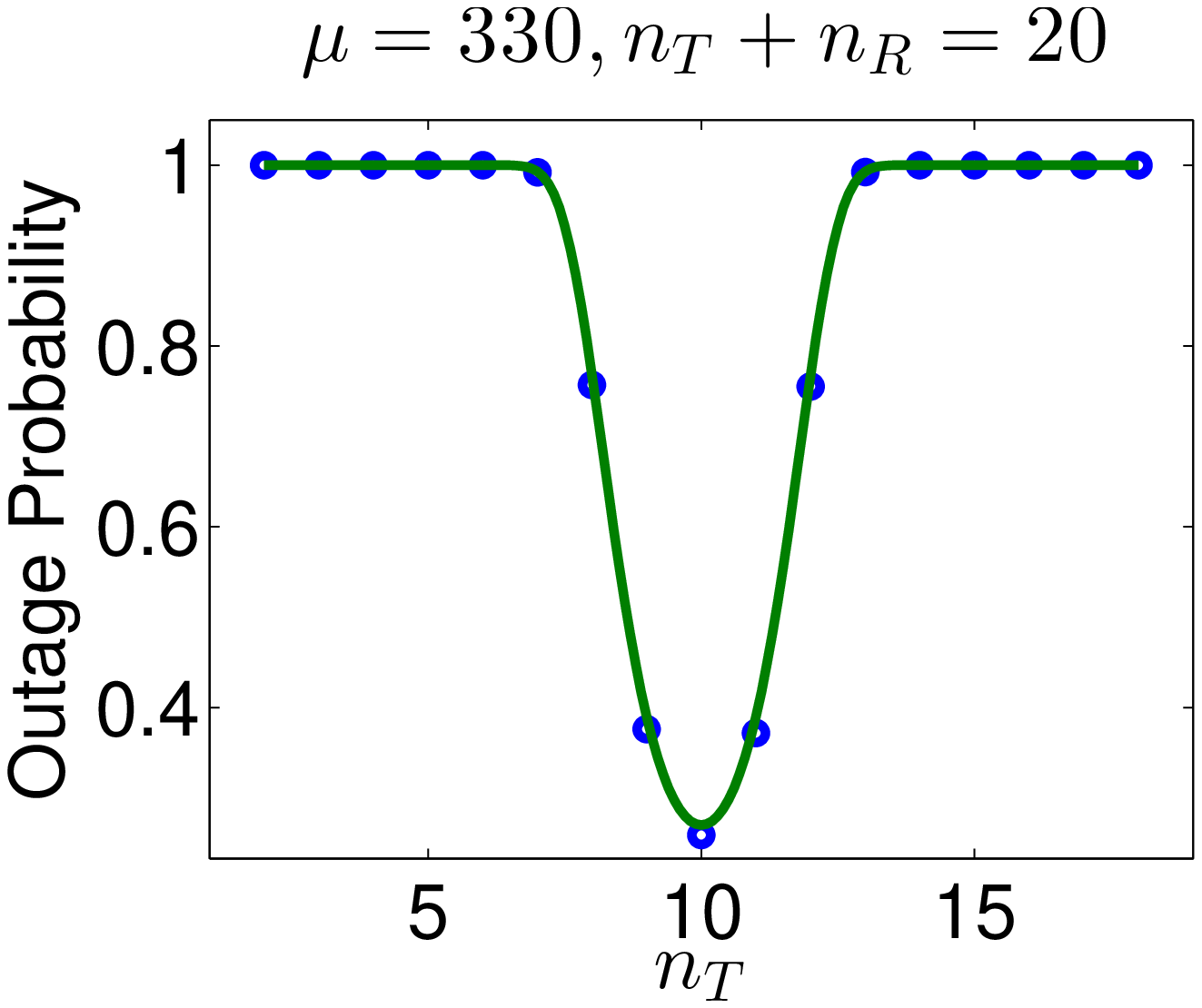}
  \caption[Const Number of Antennas]
   {\small
   The outage probability as a function of $n_T$, where $n_T+n_R$ is fixed. The circles represent a Monte-Carlo simulation whereas the solid line is our approximation (which can be computed for any non-integer $n_T \in \mathbb{R}^+$). These graphs not only support Claim \ref{claim:mimo}, but also demonstrate the accuracy of our approximations. In both graphs,  $K=2, \sigma_H=0.3, \sigma_n=1$ and $\Omega_D = 5$.
   \normalsize}
\end{figure}

%% file: simulations.tex
%----------------------------------------------------------------------------------------
\section{Simulations}
\label{sec:simulations}
%----------------------------------------------------------------------------------------

Here we assess the accuracy of our proposed approximation by a series of simulations. We calculate the empirical distribution of the largest eigenvalue using Mont Carlo realizations and compare it to our Propositions. Results for cases 1 and 2 are shown in Figure \ref{fig:prop1_2}, where for comparison, we also plot the standard  Gaussian density. Results for cases 3 and 4 are shown in Figure \ref{fig:prop3_4}. As we can see, in all cases, for small sample size and dimension, the distribution of the largest root deviates significantly from the asymptotic Gaussian one, with our propositions being able to capture this key factor. 

Figure \ref{fig:prop5} shows the accuracy of Proposition \ref{canonical}. A good match between the theoretical approximate result and simulation results is clearly visible, particularly, at the tail of the distribution.  The proposed simple stochastic characterization of the inner product between the leading sample and population eigenvectors is corroborated by the simulation results given in Figure \ref{fig:prop6_7}.

 %ADDITIONAL SIMULATIONS - IN PROGRESS
\begin{figure}[t]
\centering
\includegraphics[width=2.85in]{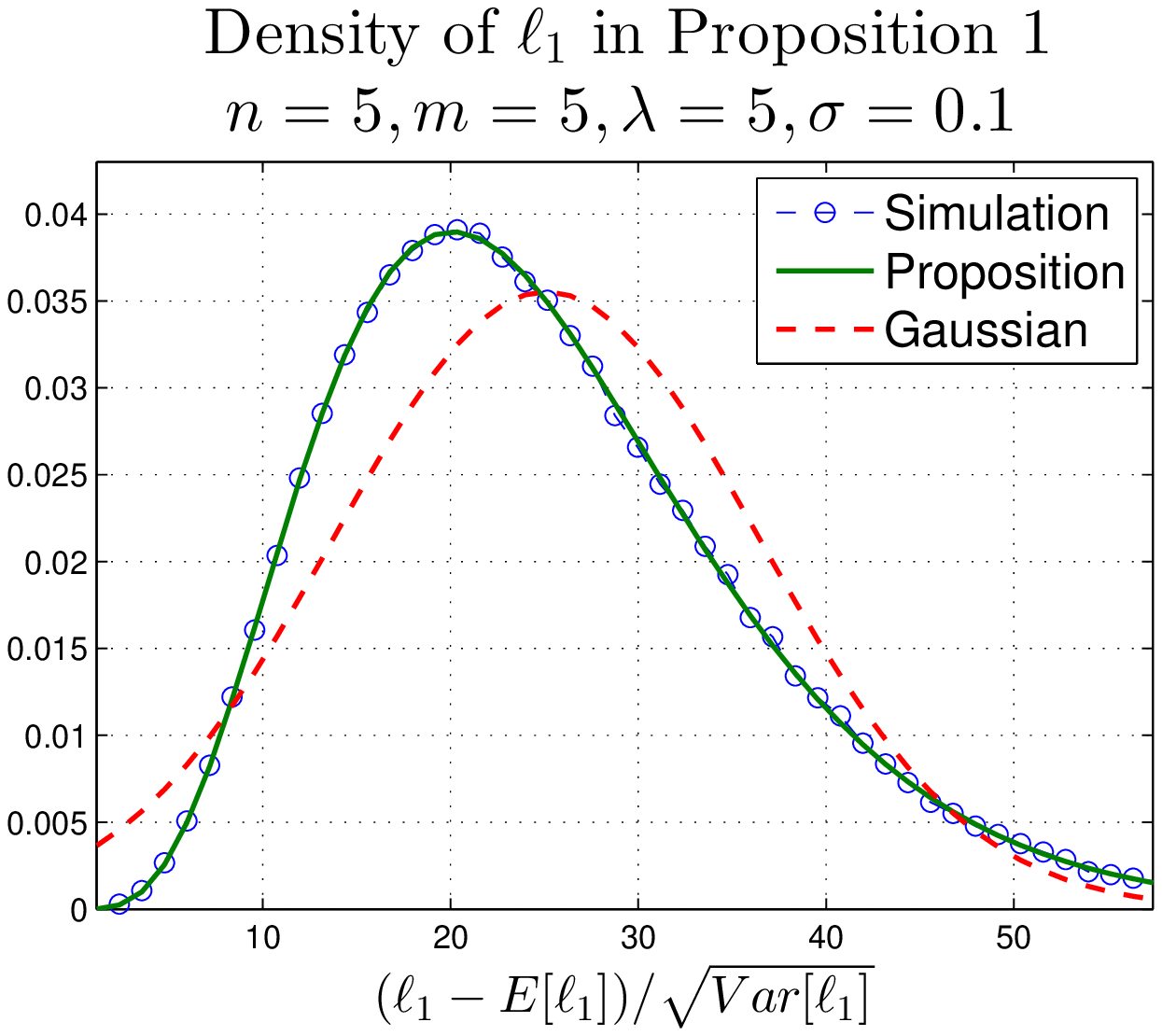}
\includegraphics[width=2.85in]{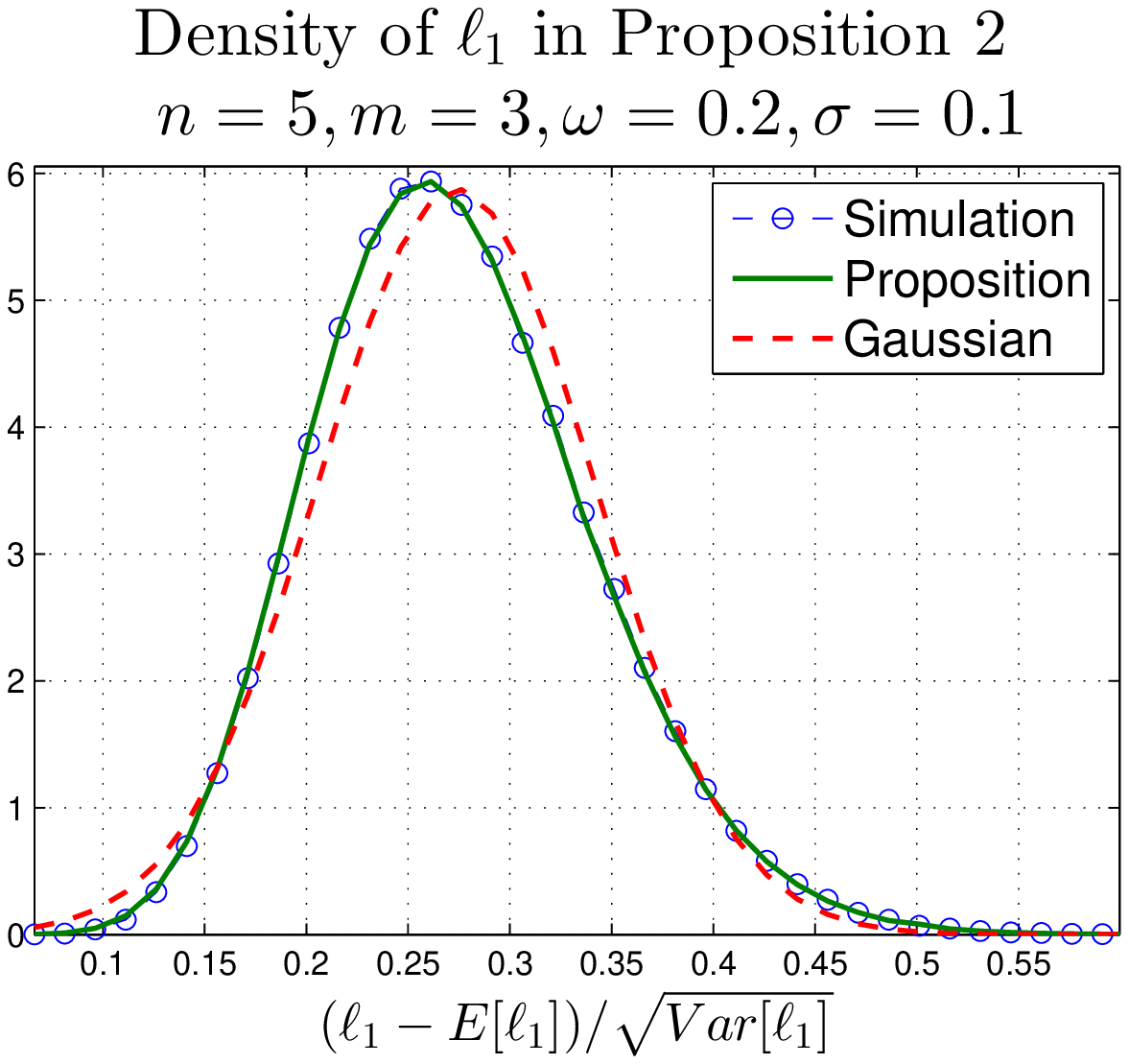}
\caption{Density functions of the largest eigenvalue in cases 1 and 2. }
\label{fig:prop1_2}
\end{figure}

\begin{figure}[t]
\centering
\includegraphics[width=2.80in]{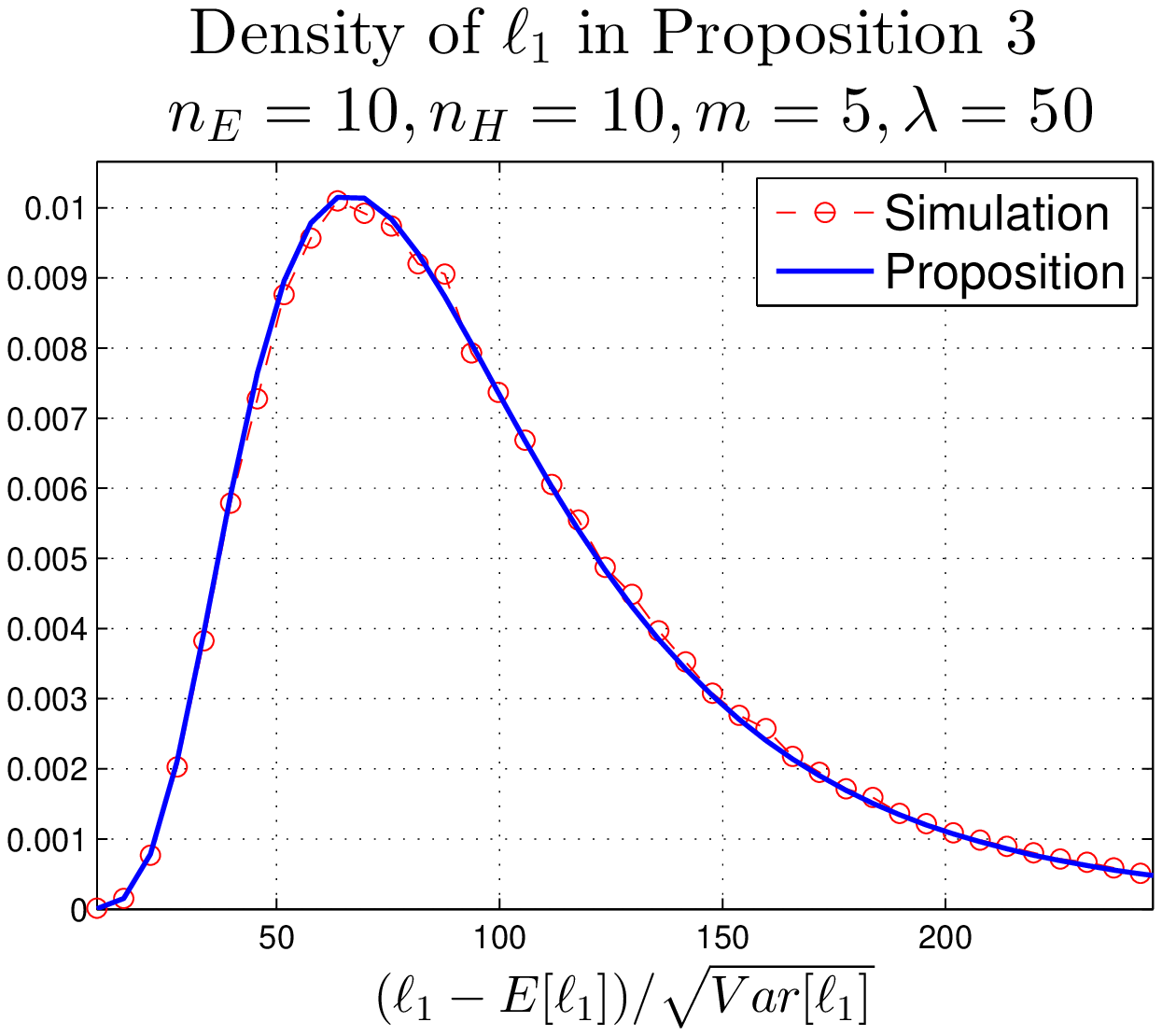}
\includegraphics[width=2.80in]{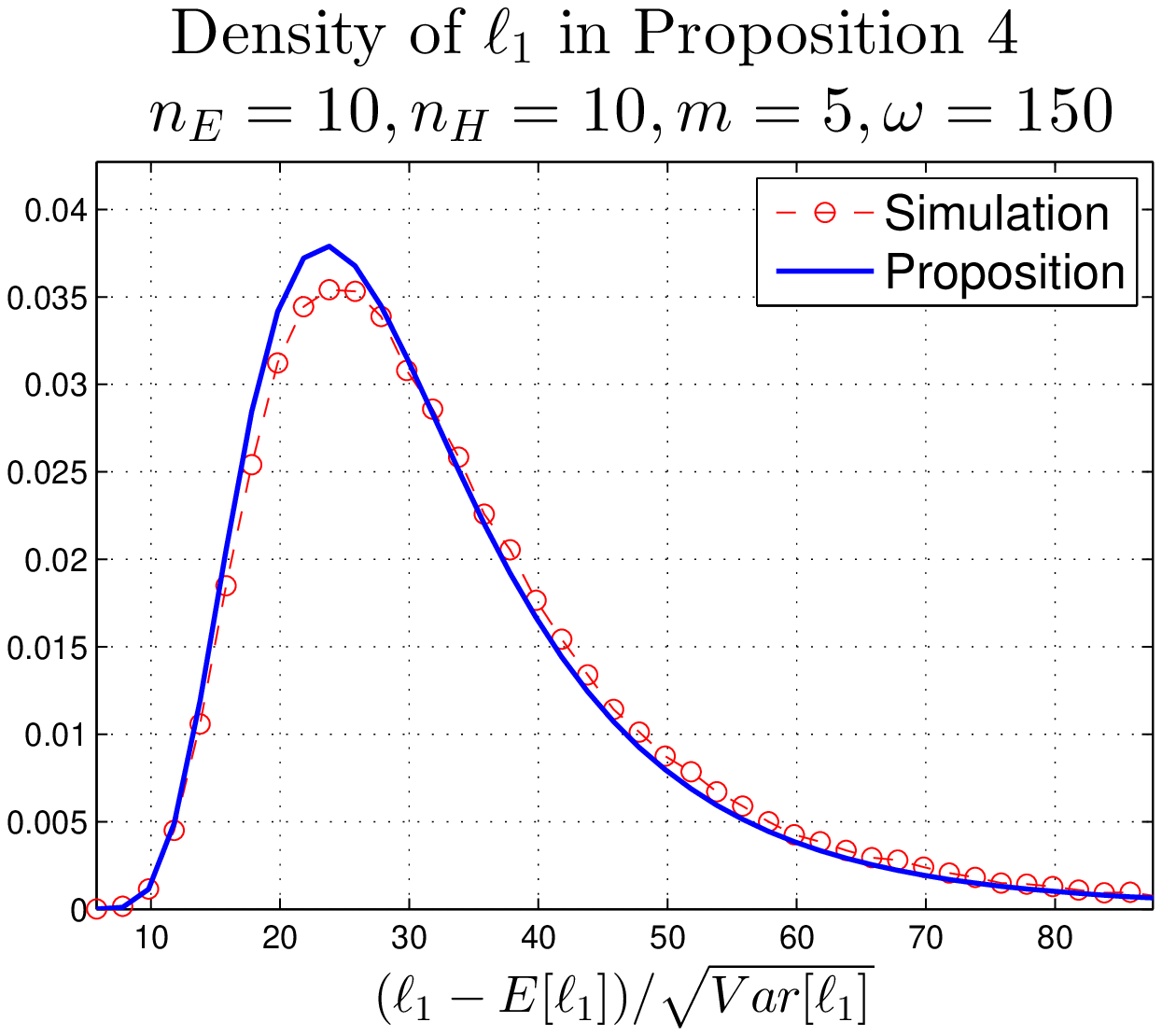}
\caption{Density functions of the largest eigenvalues in scenarios 3 and 4. }
\label{fig:prop3_4}
\end{figure}

\begin{figure}[h]
\centering
\includegraphics[width=3.5in]{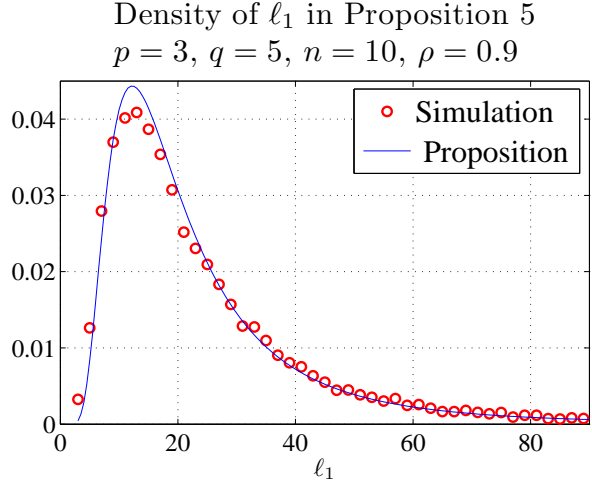}
\caption{Density function of $\ell_1(E^{-1}H)$ }
\label{fig:prop5}
\end{figure}

\begin{figure}[h]
\centering
\includegraphics[width=2.85in]{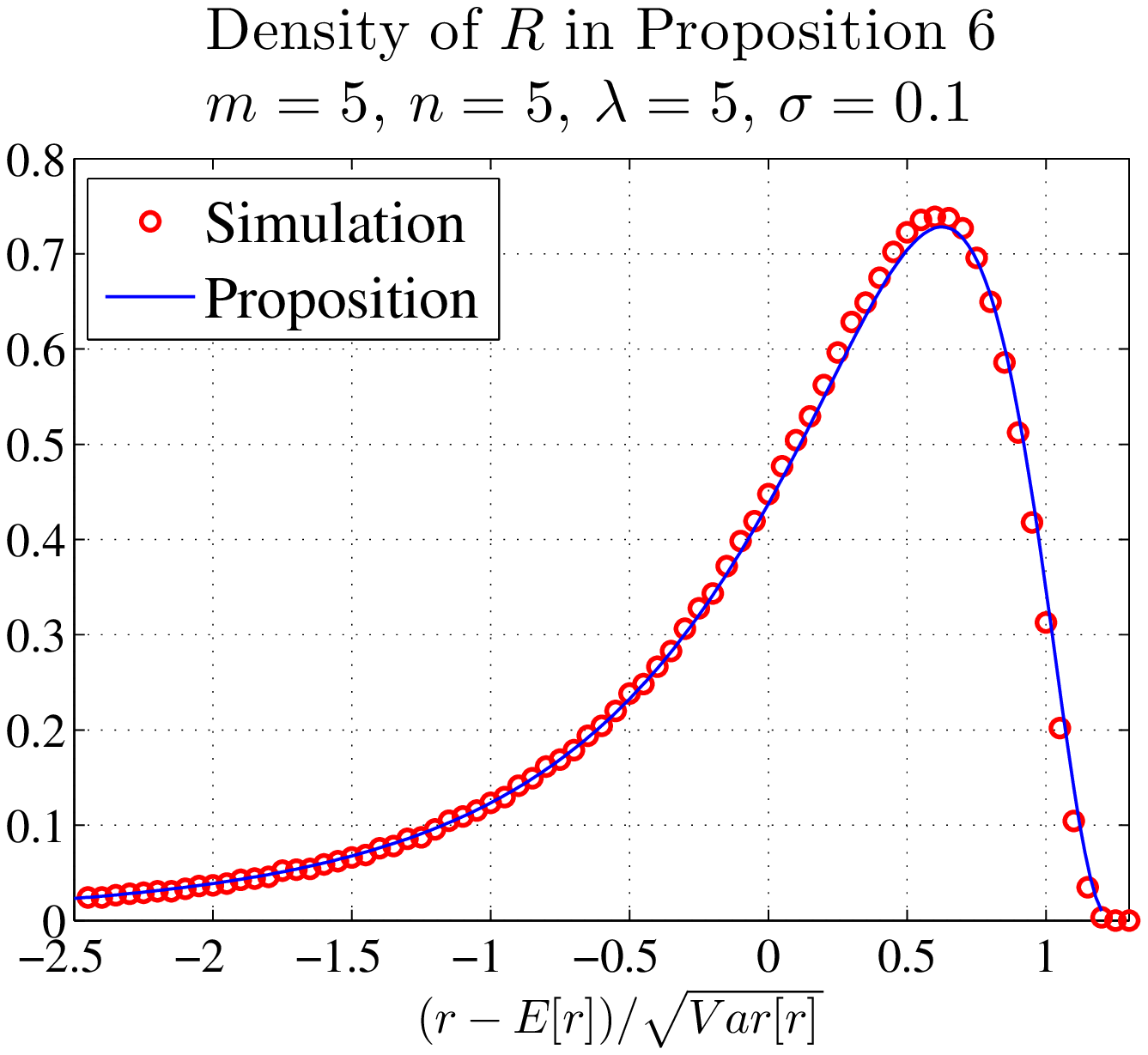}
\includegraphics[width=2.84in]{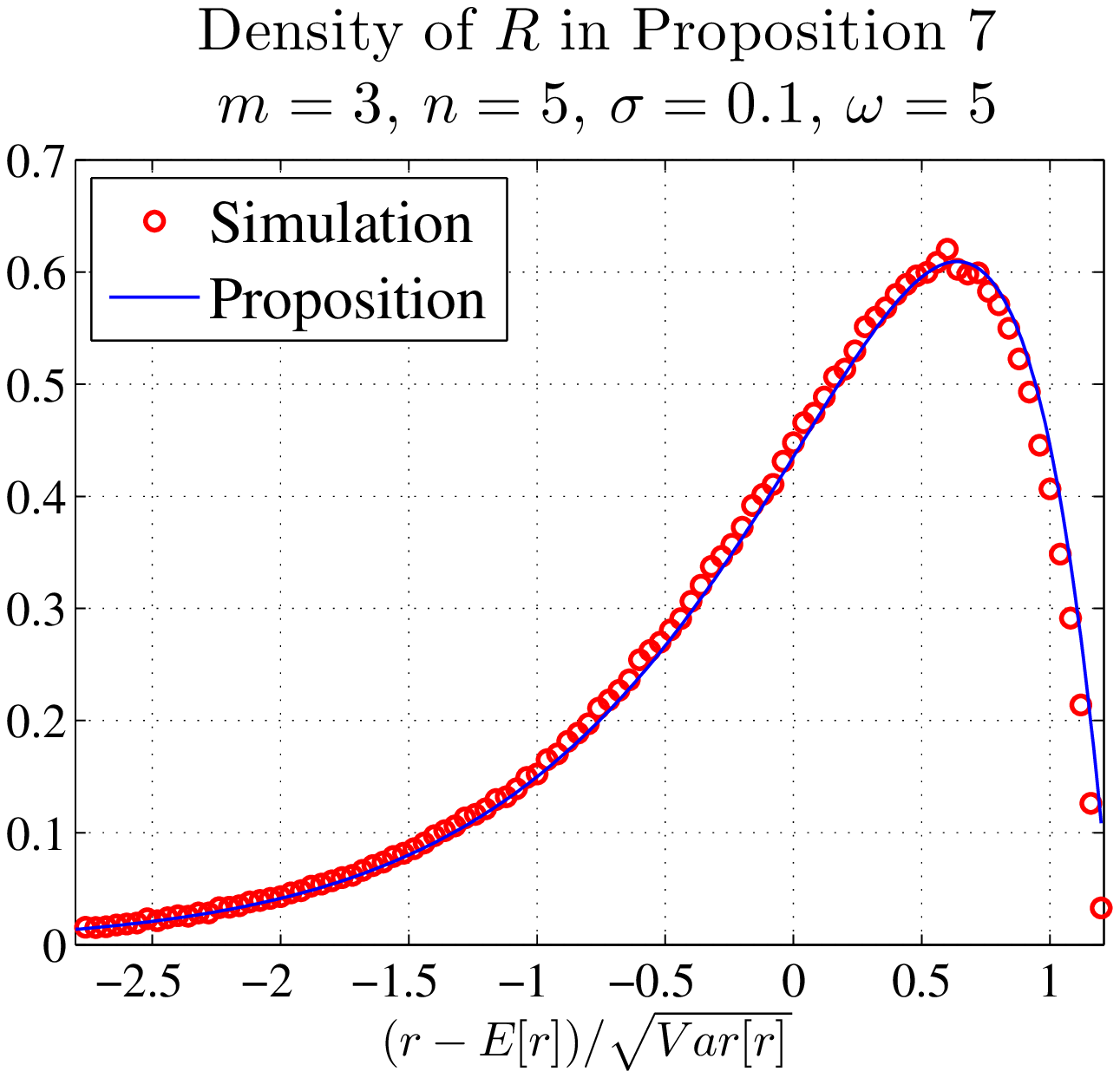}
\caption{Density functions of $R$ corresponding to scenarios 1 and 2. }
\label{fig:prop6_7}
\end{figure}

%% file: appendixA.tex
%----------------------------------------------------------------------------------------
\section{Proofs}
%----------------------------------------------------------------------------------------

To prove Prop. 1 and 2, we first present an auxiliary lemma, whose proof is provided later on in appendix \ref{sec:proofs_lemmas}.

\begin{lemma}\label{lemma:1}
Let $\{{\bf x}_j\}_{j=1}^n$ be \(n\) vectors in $\mathbb{C}^m$ of the form
\begin{equation}
\label{eq:L_1_def1}
{\bf x}_j = u_j {\bf e}_1 + \epsilon \xi_j ^ \bot
\end{equation}
with vectors $\xi_j ^ \bot= \begin{pmatrix} 0 \\ \xi_j \end{pmatrix} $, $\xi_j \in \mathbb{C}^{m-1}$.  Define a scalar $z \in \mathbb{R}$, a vector $b \in \mathbb{C}^{m-1}$ and a matrix $Z\in \mathbb{C}^{(m-1)\times(m-1)}$ as follows:
\begin{equation}
\label{eq:L_1_def2}
z = \sum_{j=1}^{n} u_j \overline{u_j}, \quad b = z^{-\frac{1}{2}} \sum_{j=1}^{n} \overline{u_j} \xi_j, \quad Z = \sum_{j=1}^{n} \xi_j \xi_j^\dagger.
\end{equation}
Finally, let $\ell_1(\epsilon)$ be the largest eigenvalue of $\ds H(\epsilon)= \sum_{j=1}^n {\bf x}_j {{\bf x}_j}^\dagger$. Then $\ell_1(\epsilon)$ is an even analytic function of $\epsilon$ and its Taylor expansion around $\epsilon=0$ is
\begin{equation}
\label{eq:L_1_l_taylor}
\ell_1(\epsilon) = z + b^\dagger b \epsilon^2 + z^{-1} b^\dagger (Z-b b^\dagger)b \epsilon ^4 + \dots
\end{equation}
\end{lemma}

\begin{proof}[Proof of Prop. 1 and 2]
First, note that the eigenvalues of $H$ are invariant under unitary transformations. Hence, w.l.g. we can assume that ${\bf v}={\bf e}_1$. Thus the matrix $H$ may be realized from $n$ i.i.d. observations of the form (\ref{eq:L_1_def1}) with $\epsilon$ replaced by $\sigma$,
\begin{equation}
\label{eq:Prop_1_2_P_dist1}
\xi_j \sim CN(0, I_{m-1}), \qquad
u_j \sim \begin{cases}
    CN(0, \sigma^2 + \lambda) & \text{Prop. 1} \\
    CN(\mu_j, \sigma^2) & \text{Prop. 2}
  \end{cases}
\end{equation}
and $\mu_j$ are arbitrary complex numbers satisfying $\sum |\mu_j|^2 = \omega$.

Lemma $1$ yields the series approximation (\ref{eq:L_1_l_taylor})  for each realization of $u=(u_k)$ and $\Xi=[\xi_1, \dots, \xi_n] \in \mathbb{C}^{(m-1) \times n}$. To see the implications of the distributional assumptions 
(\ref{eq:Prop_1_2_P_dist1}), we first rewrite (\ref{eq:L_1_l_taylor}) as follows.
Define $o_1 = \overline{u} / ||u|| \in \mathbb{C}^n$ and then choose columns $o_2,\dots,o_n$ so that $O=[o_1,\dots,o_n]$ is an $n \times n$ unitary matrix. Let $V = \Xi O$ denote the $(m-1) \times n$ matrix,  whose first column satisfies $v_1=\Xi \overline{u} / ||u|| = b$.
In this notation, the $O(\epsilon^2)$ term in Eq. (\ref{eq:L_1_l_taylor}) can be written as $b^\dagger b = ||v_1|| ^ 2$. For the forth order term, observe that $Z= \Xi \Xi^\dagger = V V^\dagger$ and so
\begin{eqnarray} 
D = b^\dagger(Z- b b^\dagger)b &=& v_1^\dagger(V V^\dagger - v_1 v_1^\dagger)v_1 = (v_1^\dagger V) (v_1^\dagger V)^\dagger - (v_1^\dagger v_1) (v_1^\dagger v_1)^\dagger \nonumber\\
 &=& \sum_{j=2}^n |v_1^\dagger v_j|^2. \nonumber 
\end{eqnarray}
Hence, Eq. (\ref{eq:L_1_l_taylor})  becomes
$$\ell_1(\epsilon) = V_{0} + V_2 \epsilon^2 + V_4 \epsilon^4 + \dots$$
where
\begin{align*}
V_0 = ||u||^2, \qquad V_2 = ||v_1||^2, \qquad V_4 = V_0^{-1}D.
\end{align*}
Now, we bring in the distributional assumptions (\ref{eq:Prop_1_2_P_dist1}) in order to study the distributions of $V_0, V_2, V_4$.

Observe that $u_j = \frac{1}{\sqrt{2}}a_j + \frac{i}{\sqrt{2}} b_j$, where
\begin{equation*}
a_j \sim \begin{cases}
    N(0, \lambda+\sigma^2) & \text{Prop. 1} \\
    N(\sqrt{2}Real(\mu_j), \sigma^2) & \text{Prop. 2}
  \end{cases}
  \qquad
b_j \sim \begin{cases}
  N(0, \lambda+\sigma^2) & \text{Prop. 1} \\
  N(\sqrt{2}Im(\mu_j), \sigma^2) & \text{Prop. 2}
\end{cases}
\end{equation*}
so $\ds ||u||^2= \frac{1}{2}\sum_{j=1}^n ( a_j ^2 + b_j ^2)$ is a sum of $2n$ independent squares of Gaussian random variables, and therefore
\begin{equation*}
V_0 = ||u||^2 \sim \begin{cases}
    \frac{\sigma^2 + \lambda}{2} \chi_{2n}^2 & \text{Prop. 1} \\
    \frac{\sigma^2}{2} \chi_{2n}^2(\frac{2 \omega}{\sigma^2}) & \text{Prop. 2}.
  \end{cases}
\end{equation*}
Since $O$ is unitary and fixed once $u$ is given, the columns $v_j|u \sim \mathcal C\mathcal N(0, I_{m-1})$. The distribution of $v_j$ does not depend of $u$, hence $v_j \sim \mathcal C\mathcal N(0, I_{m-1}) $. Applying the same arguments as before
$$V_2 = ||v_1||^2 \sim \frac{1}{2}\chi_{2m-2}^2$$
independently of $||u||^2$.\\
Finally, conditioned on $(u, v_1)$, we have $v_1^\dagger v_j \sim \mathcal C\mathcal N(0, ||v_1||^2)$ and $|v_1^\dagger v_j|^2 \sim \frac{||v_1||^2}{2} \chi_2^2 $. Again, applying the same arguments as before,
$$D|(u, v_1) = \sum_{j=2}^n |v_1^\dagger v_j|^2 |(u, v_1) \sim \frac{||v_1||^2}{2} \chi_{2n - 2}^2 $$
where the $\chi_{2n-2}^2$ variate is independent of $(u, v_1)$.

We conclude that
\begin{equation*}
V_4  \sim \begin{cases}
    \frac{1}{2\sigma^2 + 2\lambda} (\chi_{2n}^2)^{-1} \chi_{2m - 2}^2 \chi_{2n - 2}^2 & \text{Prop. 1} \\
    \frac{1}{2\sigma^2} (\chi_{2n}^2(\frac{2 \omega}{\sigma^2}))^{-1} \chi_{2m - 2}^2 \chi_{2n - 2}^2 & \text{Prop. 2}
  \end{cases}
\end{equation*}
and this completes the proof of Propositions 1 and 2.
\end{proof}

To prove Propositions 3 and 4, we first introduce some additional notation and two auxiliary lemmas, whose proofs are provided later on in appendix \ref{sec:proofs_lemmas}.
For a matrix $S$, denote by $S_{jk}$ and $S^{jk}$ the $(j,k)$-th entry of $S$ or $S^{-1}$, respectively.
 
%----------------------------------------------------------------------------------------
\begin{lemma}\label{lemma:2}
Let $E \sim \mathcal{CW}_m(n, I)$ and $M = [{\bf e}_1,  b] \in \mathbb{C}^{m\times2} $, with $ b \bot {\bf e}_1$ fixed. Then
\begin{equation}
\label{eq:L_2_dist1}
S = (M^\dagger E^{-1} M)^{-1} \sim CW_2(n-m+2, D), \qquad D = diag(1, \frac{1}{||b||^2})
\end{equation}
and the two random variables $S^{11}$ and $S_{22}$ are independent with
\begin{equation}
\label{eq:L_2_dist2}
S^{11} \sim  \frac{2}{\chi_{2n - 2m + 2}^2}, \qquad S_{22} \sim  \frac{\chi_{2n - 2m + 4}^2}{2||b||^2}.
\end{equation}
\end{lemma}

\begin{lemma}\label{lemma:3}
Let $E \sim \mathcal{CW}_m(n, I)$ and let
$A_2=\begin{pmatrix} 0 & 0 \\ 0 & Z\end{pmatrix}$
where $Z$ is a $(m-1)\times(m-1)$ random matrix independent of $E$, with $\mathbb{E}[Z] = I_{m-1}$. Then
\begin{equation}
\label{eq:L_3_expt}
\mathbb{E}\left[ \frac{ {\bf e}_1^T E^{-1} A_2 E^{-1} {\bf e}_1 }{E^{11}}  \right] = \frac{m-1}{(n-m)(n-m+1)}.
\end{equation}
\end{lemma}

\begin{proof}[Proof of Prop. 3 and 4]

Without loss of generality we may assume that the signal direction is ${\bf v}={\bf e}_1$. Hence
$$
H \sim \begin{cases}
    \mathcal{CW}_m(n_H, I_m + \lambda {\bf e}_1 {\bf e}_1^T) & \text{Prop. 3} \\
    \mathcal{CW}_m(n_H, I_m, \omega {\bf e}_1 {\bf e}_1^T) & \text{Prop. 4}.
  \end{cases}
$$
Next, we apply a perturbation approach similar to the one used in the previous proof. To introduce a small parameter, set
$$
\epsilon ^ 2 = \begin{cases}
    1/(1+\lambda) & \text{Prop. 3} \\
    1/\omega & \text{Prop. 4}.
  \end{cases}
$$
The matrix $H_\epsilon = \epsilon^2H$ has a representation of the form $X^\dagger X$ with $X=[{\bf x}_1, \dots, {\bf x}_{n_H}]$ where ${\bf x}_j$ are of the form (\ref{eq:L_1_def1}) but now with
\begin{equation}
\label{eq:Prop_3_4_P_dist1}
\xi_j \sim \mathcal C\mathcal N(0, I_{m-1}), \qquad
u_i \sim \begin{cases}
    \mathcal C\mathcal N(0, 1) & \text{Prop. 3} \\
    \mathcal C\mathcal N(\mu_j/\sqrt{\omega} , 1/\omega) & \text{Prop. 4}
  \end{cases}
\end{equation}
where $\sum |\mu_j|^2 = \omega$. In particular,
\begin{equation}
\label{eq:Prop_1_2_P_dist2}
z = \sum_{j=1}^{n_H} |u_j|^2
 \sim \begin{cases}
      \frac{1}{2} \chi_{2n_H}^2 & \text{Prop. 3} \\
      \frac{1}{2\omega}\chi_{2n_H}^2(2\omega) & \text{Prop. 4}.
  \end{cases}
\end{equation}

With $b$ as in (\ref{eq:L_1_def2}), using the same arguments as in the previous proof, we have that $b \sim \mathcal{CN}(0, I_{m-1})$, independently of $u$.

The matrix $H_\epsilon$ has a decomposition in the form $H_\epsilon=A_0 + \epsilon A_1 + \epsilon^2 A_2$, where
$$
A_0 = \begin{pmatrix}z & 0 \\ 0 & 0_{n-1} \end{pmatrix}, \qquad
   A_1 = \sqrt{z}\begin{pmatrix}0 & b^\dagger \\ b & 0_{n-1} \end{pmatrix}, \qquad
   A_2 = \begin{pmatrix}0 & 0 \\ 0 & Z \end{pmatrix}. $$
with $Z$ as in (\ref{eq:L_1_def2}).

For future use we define the following quantities
\begin{equation}
\label{eq:Prop_1_2_Es_def}
E^{11} = {\bf e}_1^T E^{-1} {\bf e}_1, \qquad
E^{b1} = {\hat b}^\dagger E^{-1} {\bf e}_1, \qquad
E^{bb} = {\hat b}^\dagger E^{-1} \hat b
\end{equation}
where $\hat b = \begin{pmatrix} 0 \\ b \end{pmatrix}$. Note that the condition $n_E> m$ ensures that $E$ is invertible with probability 1.

Since $E^{-1/2} H_{\epsilon} E^{-1/2}$ is Hermitian for all \(\epsilon\), the largest eigenvalue \(\ell_1(\epsilon)\) is real-valued. Furthermore, since $E^{-1/2} H_{\epsilon} E^{-1/2}$ is an holomorphic symmetric function of $\epsilon$, it follows from Kato (\cite{Kato}, Theorem 6.1 page 120) that the largest eigenvalue $\ell_1$ and its eigenprojection $P'(\epsilon)$ are analytic functions of $\epsilon$ in some neighbourhood of zero. The eigenvalues of $E^{-1}H_{\epsilon}$ are the same as those of the matrix $E^{-1/2} H_{\epsilon} E^{-1/2}$, therefore the largest eigenvalue of $E^{-1}H_{\epsilon }$ is also an analytic function of $\epsilon$. The projection to the corresponding eigenspace of $E^{-1}H_{\epsilon}$ is $P(\epsilon) = E^{-1/2}P'(\epsilon)$. Since the matrix $E$ does not depend on $\epsilon$, this projection is also an analytic function in some neighborhood of zero.

For $\epsilon=0$, $E^{-1}{\bf e}_1$ is an eigenvector with eigenvalue $E^{11}z$, that is,
$$E^{-1}H_0E^{-1}\textbf{e}_1 = zE^{-1}\textbf{e}_1\textbf{e}_1^TE^{-1}\textbf{e}_1=zE^{11}E^{-1}\textbf{e}_1$$.
Hence,
\begin{equation}
\label{Prop_3_4_e1_innerP}
\langle P(0)E^{-1} {\bf e}_1, {\bf e}_1\rangle =\langle E^{-1} {\bf e}_1, {\bf e}_1\rangle = E^{11}.
\end{equation}
Since $P$ is an analytic function of \(\epsilon\) and the inner product is a smooth function, then   $\langle P(\epsilon)E^{-1}{\bf e}_1, {\bf e}_1\rangle $ is an analytic non-zero function in some neighborhood of $\epsilon=0$. Thus, we may define
\begin{equation}
\label{eq:Prop_3_4_v_1_def}
v_1(\epsilon)  =E^{11} \langle P(\epsilon)E^{-1} {\bf e}_1, {\bf e}_1\rangle ^{-1}  P(\epsilon)E^{-1} {\bf e}_1
\end{equation}
Clearly $v_1(\epsilon)$ is an eigenvector corresponding to the eigenvalue  $\ell_1(\epsilon)$ and it is also analytic in some neighbourhood of zero.
We thus expand
\begin{equation}
\label{eq:Prop_3_4_P_seq}
 \ell_{1}(\epsilon) = \sum_{j=0}^{\infty} \lambda_j \epsilon^j, \qquad
v_1(\epsilon) = \sum_{j=0}^{\infty} w_j \epsilon^j.
\end{equation}
Inserting these expansions into the eigenvalue-eigenvector equations $E^{-1}H_{\epsilon}v_1 = \ell_{1} v_1$, we get the following equations:
At the $O(1)$ level,
\begin{equation}
E^{-1}A_0 w_0 = \lambda_0 w_0
\end{equation}
whose solution is
\begin{equation}
\label{eq:Prop_3_4_P_ev_eq_1}
\lambda_0 = z E^{11}, \qquad
w_0 = const \cdot E^{-1}{\bf e}_1.
\end{equation}
Using equations (\ref{Prop_3_4_e1_innerP})-(\ref{eq:Prop_3_4_v_1_def}), we conclude that $w_0 = v_1(0) = E^{-1}{\bf e}_1$, meaning the normalization constant is one.

From Eq. (\ref{eq:Prop_3_4_v_1_def}) it follows that ${\bf e}_1^T v_1(\epsilon) = E^{11} = {\bf e}_1^T w_0$. Hence ${\bf e}_1^T w_j = 0$ for all $j \geq 1$. Furthermore, since $A_0 = z {\bf e}_1 {\bf e}_1^T$, this normalization also conveniently gives us that $A_0w_j = 0$ for all $j \geq 1$.\\
The $O(\epsilon)$ equation is
\begin{equation}
\label{eq:Prop_3_4_P_ev_eq_2}
E^{-1}A_1 w_0 + E^{-1}A_0 w_1 = \lambda_1 w_0 + \lambda_0 w_1.
\end{equation}
However, $A_0w_1 = 0$. Multiplying this equation by ${\bf e}_1^T$ gives that
\begin{eqnarray}
\lambda_1 &=& \frac{{\bf e}_1^T E^{-1} w_0}{E^{11}} = \frac{\sqrt{z}}{E^{11}}({\bf e}_1^T E^{-1} \begin{pmatrix}0 & b^\dagger \\ b & 0 \end{pmatrix} E^{-1}{\bf e}_1) =\\
&=&\frac{\sqrt{z}}{E^{11}}({\bf e}_1^T E^{-1}\begin{pmatrix}0 & b^\dagger \\ 0 & 0 \end{pmatrix}E^{-1}{\bf e}_1 + {\bf e}_1^T E^{-1}\begin{pmatrix}0 & 0 \\ b & 0 \end{pmatrix}E^{-1}{\bf e}_1)= \nonumber
\\
&=&\frac{\sqrt{z}}{E^{11}}({\bf e}_1^T E^{-1}\begin{pmatrix}0 & b^\dagger \\ 0 & 0 \end{pmatrix}E^{-1}{\bf e}_1 + {\bf e}_1^T E^{-1}\begin{pmatrix}0 & b \\ 0 & 0 \end{pmatrix}E^{-1}{\bf e}_1)\nonumber=
\\
&=& 2\sqrt{z}Re(E^{b1}).\nonumber
\end{eqnarray}
Inserting the expression for $\lambda_1$ into Eq. (\ref{eq:Prop_3_4_P_ev_eq_2}) gives that
\begin{eqnarray}\\\nonumber
w_1 &=& \frac{1}{\sqrt{z}E^{11}}(E^{-1} \begin{pmatrix}0 & b^\dagger \\ b & 0 \end{pmatrix}E^{-1} {\bf e}_1 - 2 Re(E^{b1})E^{-1}{\bf e}_1)=
\\
&=&\frac{1}{\sqrt{z}E^{11}}(E^{-1} \begin{pmatrix}0 & b^\dagger \\ 0 & 0 \end{pmatrix}E^{-1} {\bf e}_1 + E^{-1} \begin{pmatrix}0 & 0 \\ b & 0 \end{pmatrix}E^{-1} {\bf e}_1 - 2 Re(E^{b1})E^{-1}{\bf e}_1)=\nonumber
\\
&=&\frac{1}{\sqrt{z}E^{11}}(E^{b1} E^{-1} {\bf e}_1 + E^{11}E^{-1}\hat b - 2 Re(E^{b1})E^{-1}{\bf e}_1)=\nonumber
\\
&=&\frac{1}{\sqrt{z}} \left(E^{-1}\hat b - \frac{\overline{E^{b1}}}{E^{11}}E^{-1}{\bf e}_1\right).\nonumber
\end{eqnarray}
The next $O(\epsilon^2)$ equation is
\begin{equation}
E^{-1}A_2w_0 + E^{-1}A_1 w_1 + E^{-1}A_0w_2= \lambda_2 w_0 + \lambda_1 w_1 + \lambda_0 w_2.
\end{equation}
Multiplying this equation by ${\bf e}_1^T$ and recalling that $A_0 w_2 = 0$ and that ${\bf e}_1^\dagger w_0 = E^{11}$ gives
\begin{eqnarray}
\label{eq:Prop_3_4_P_ev_eq_3}
\lambda_2 &=& \frac{{\bf e}_1^T E^{-1} A_2 E^{-1} {\bf e}_1}{E^{11}} + \frac{{\bf e}_1^T E^{-1} A_1 \frac{1}{\sqrt{z}} (E^{-1}\hat b - \frac{\overline{E^{b1}}}{E^{11}}E^{-1}{\bf e}_1)}{ E^{11} }=
\\
&=&\frac{{\bf e}_1^T E^{-1} A_2 E^{-1} {\bf e}_1}{E^{11}} + \frac{E^{11}E^{bb} + (\overline{E^{b1}})^2 - 2\overline{E^{b1}}Re(E^{b1})}{E^{11}}=\nonumber
\\
&=&\frac{{\bf e}_1^T E^{-1} A_2 E^{-1} {\bf e}_1}{E^{11}} + \frac{E^{11}E^{bb} - E^{b1}\overline{E^{b1}}}{E^{11}}.\nonumber
\end{eqnarray}
Combining Eqs. (\ref{eq:Prop_3_4_P_ev_eq_1})-(\ref{eq:Prop_3_4_P_ev_eq_3}), we obtain the following approximate stochastic representation for the largest eigenvalue $\ell_1$ of $E^{-1}H_{\epsilon}$
\begin{equation}
\label{eq:Prop_3_4_ell_long}
\ell_1(\epsilon) = zE^{11} + 2\epsilon \sqrt{z}Re(E^{b1}) + \epsilon^2 \frac{{\bf e}_1^T E^{-1} A_2 E^{-1} {\bf e}_1}{E^{11}} + \epsilon^2 \frac{E^{11}E^{bb} - E^{b1}\overline{E^{b1}}}{E^{11}} + o(\epsilon^2).
\end{equation}

Next, to derive the approximate distribution of $\ell_1$ corresponding to the above equation, we study a $2 \times 2$ Hermitian matrix $S$, whose inverse is defined by
\begin{equation}
\label{eq:Prop_3_4_S_def}
S^{-1} =
\begin{pmatrix}
E^{11} & \overline{E^{b1}} \\
E^{b1} & E^{bb}
\end{pmatrix}
\end{equation}
meaning  $S^{-1} = M^\dagger E^{-1} M$, where the $m \times 2$ matrix $M = [{\bf e}_1 , \hat b]$. Inverting this matrix gives
\begin{equation}
S = \frac{1}{E^{11}{E^{bb} - E^{b1} \overline{E^{b1}} }}
\begin{pmatrix}
E^{bb} & -E^{b1} \\
-\overline{E^{b1}} & E^{11}.
\end{pmatrix}
\end{equation}
Hence in terms of the matrix $S$ and $S^{-1}$, Eq. (\ref{eq:Prop_3_4_ell_long}) can be written as
\begin{equation}
\label{eq:Prop_3_4_ell_eq1}
\ell_1(\epsilon) = zS^{11} + 2\epsilon \sqrt{z}Re(E^{b1}) + \frac{\epsilon^2}{S_{22}} + \epsilon^2 \frac{{\bf e}_1^T E^{-1} A_2 E^{-1} {\bf e}_1}{E^{11}} + o(\epsilon^2).
\end{equation}

To establish Propositions 3 and 4, we start from Eq. (\ref{eq:Prop_3_4_ell_eq1}). We neglect the second term $T_1 = 2 \epsilon \sqrt{z} Re(E^{b1})$ which is symmetric with mean zero, and whose variance is much smaller than that of the first term. We also approximate the last term, denoted by $T_2$, by its mean value, using Lemma 3. We now have
$$ \ell_1(\epsilon) \approx zS^{11} + \frac{\epsilon^2}{S_{22}} + \epsilon^2 c(m, n) $$
where $c(m, n)$ is the expectation from Lemma 3.
Recall that (\ref{eq:Prop_3_4_ell_eq1}) is the largest eigenvalue of $E^{-1}H_\epsilon = \epsilon^2E^{-1}H$. We need to divide by $\epsilon^2$ in order to get the eigenvalue of $E^{-1}H$.
By doing so, and inserting the distributions of $S^{11}, S_{22}$, that are known from Lemma 2,  we have
\begin{align*}
\ell_1(E^{-1}H) \approx \frac{2z}{\epsilon^2 \chi_{2n_E-2m+2}^2} + \frac{2||b||^2}{\chi_{2n_E-2m+4}^2} + \frac{m-1}{(n_E -m)(n_E -m -1)} .
\end{align*}
Next, by inserting the distributions of $||b||^2$, $z$ and the relevant value of $\epsilon$, we get that for Proposition 3
$$ \ell_1(\lambda) \approx (1+\lambda)
\frac{\chi_{2n_H}^2}{\chi_{2n_E-2m+2}^2} +
\frac{\chi_{2m-2}^2}{\chi_{2n_E-2m+4}^2} +
\frac{m-1}{(n_E -m)(n_E -m -1)} $$
and for Proposition 4
\begin{align*}
 \ell_1(\omega) \approx
\frac{\chi_{2n_H}^2(2\omega)}{ \chi_{2n_E-2m+2}^2} +
\frac{\chi_{2m-2}^2 }{\chi_{2n_E-2m+4}^2} +
\frac{m-1}{(n_E -m)(n_E -m -1)}.
\end{align*}
Notice that from lemma 2 and the comment about the independency of $u$ and $z$ in the beginning of the proof, we get that all of the above $\chi^2$ random variables are independent.
At last, since ratios of independent $\chi^2$ random variables follow a $F$ distribution, the two propositions follow.

\end{proof}

\begin{proof}[Proof of Prop. \ref{canonical}]
Since $\omega$ depends through $X^\dagger X$, following (\ref{HEcondx}), we invoke Proposition \ref{prop:4} with the re-parametrization $m=p, n_H=q$, and $n_E=n-q$ to obtain the approximate conditional distribution of $l_1$ as
\begin{align}
\ell_1\left(E^{-1}H\right)|X\approx a_1 F_{b_1,c_1}\left(\frac{2\rho^2}{1-\rho^2} \left( X^\dagger X\right)_{11}\right)+a_2F_{b_2,c_2}+a_3.
\end{align} 
Since $\left( X^\dagger X\right)_{11}\sim \frac{1}{2}\chi^2_{2n}$, the final result follows by removing the condition by using the definition of $F_{a,b}^\chi(c,n) $ given in (\ref{Fchidef}).
\end{proof}

\begin{proof}[Proof of Prop. \ref{prop:innerv} and \ref{prop:innervnon}]

Let us assume w.l.o.g. that $\mathbf{v}=\mathbf{e}_1$. Therefore, we can write (\ref{R}) as
\begin{align}
R=\frac{|\hat{\mathbf{v}}^\dagger \mathbf{e}_1|^2}{||\hat{\mathbf{v}}||^2}.
\end{align}
Moreover, following Lemma 1,  we have
\begin{align}
\label{eigapprox}
\hat{\mathbf{v}}=w_0+\sigma w_1+\sigma^3 w_3+\cdots
\end{align}
where
\begin{align}
\label{eigvecdef}
w_0=\mathbf{e}_1,\;\;\; w_1=\frac{1}{||u||}\left(\begin{array}{c} 0\\ v_1\end{array}\right),\;\;\; w_3=\frac{1}{||u||^3} 
\left(\begin{array}{c} 0\\ \sum_{j=2}^n v_j v_j^\dagger v_1\end{array}\right).
\end{align}
Here
\begin{align}
u\sim \Biggl\{\begin{array}{ll}
\mathcal{CN}(0,(\lambda+\sigma^2)I_n)  & \text{for $H\sim \mathcal{CW}_m(n,\lambda \mathbf{e}_1\mathbf{e}_1^{'}+\sigma^2 I_m)$}\\
\mathcal{CN}(\mu,\sigma^2I_n)  & \text{for $H\sim \mathcal{CW}_m(n,\sigma^2I_m,(\omega/\sigma^2)\mathbf{e}_1\mathbf{e}_1^{'})$}
\end{array}
\end{align}
and $v_j\sim\mathcal{CN}(0,I_{m-1})$ are independent random vectors with $\omega=||\mu||^2$. Therefore, we have
\begin{align}
R=\displaystyle \frac{1}{\displaystyle 1+\sigma^2 \frac{||v_1||^2}{||u||^2}+2\sigma^4 \frac{\sum_{j=2}^n |v_1^\dagger v_j|^2}{||u||^4}+\cdots}.
\end{align}
The final result follows by using the distributional arguments given in the proof of Propositions 1 and 2.
\end{proof}